\documentclass[a4paper]{amsart}
\usepackage{latexsym,bm,stmaryrd}
\usepackage{amsmath,amsthm,amsfonts,amssymb,mathrsfs,pb-diagram,amscd}

\usepackage[all]{xy}

\usepackage{microtype}
\usepackage{mathtools}
\usepackage{mathbbol,wasysym}

\usepackage{tikz}
\usetikzlibrary{matrix}
\usepackage[enableskew,vcentermath]{youngtab}
\usepackage[boxsize=1.25em, centerboxes]{ytableau}
\usepackage[misc]{ifsym}

\newcommand{\Parts}[1][n]{\P_{#1}}
\newcommand\blam{{\boldsymbol\lambda}}

\newcommand\bmu{{\boldsymbol\mu}}

\DeclareMathOperator\Shape{Shape}

\let\<=\langle
\let\>=\rangle
\def\({\big(}
\def\){\big)}

\def\Z{\mathbb{Z}}

\def\C{\mathbb{C}}
\def\N{\mathbb{N}}

\def\t{\mathfrak{t}}
\def\s{\mathfrak{s}}
\def\u{\mathfrak{u}}
\def\v{\mathfrak{v}}

\def\OP{\mathcal{P}}
\def\CC{\mathscr{C}^{c}}
\def\CP{\mathscr{P}^{c}}

\def\tlam{\t^{\blam}}
\def\tmu{\t^{\bmu}}
\def\O{\mathcal{O}}
\def\D{\mathcal{D}}
\def\lam{\lambda}
\def\Lam{\Lambda}
\def\Sym{\mathfrak{S}}

\def\mL{\mathcal{L}}

\def\bn[#1,#2]{\begin{bmatrix}#1\\#2\end{bmatrix}}

\def\K{\mathscr{K}}
\def\P{\mathscr{P}}

\def\fb{\mathbf{b}}

\def\mf{\mathfrak{f}}
\def\fm{\mathfrak{m}}

\def\SStd{\mathop{\rm Std}\nolimits^2}

\def\D{\mathscr{D}}

\def\dyo[#1]{y^{\<#1\>}}

\def\tlam{\t^\blam}

\def\fb{\mathbf{b}}

\newcommand\RR{\mathscr{R}}
\newcommand\R[1][n]{\RR_{#1}^{\Lambda}}

\newcommand\HH{\mathscr{H}}

\renewcommand{\Parts}[1][n]{\P_{#1}}

\newcommand\bi{\mathbf{i}}

\DeclareMathOperator\Hom{Hom}

\DeclareMathOperator\cha{char}

\DeclareMathOperator\id{id}

\DeclareMathOperator\Tr{Tr}

\DeclareMathOperator\res{res}
\DeclareMathOperator\Ht{ht}

\DeclareMathOperator\cont{cont}
\DeclareMathOperator\Std{Std}

\title[Cocenter of the cyclotomic Hecke algebras]{On the cocenter of the cyclotomic Hecke algebra of type $G(r,1,n)$}
\author{Jun Hu}\address{Key Laboratory of Algebraic Lie Theory and Analysis of Ministry of Education\\
School of Mathematics and Statistics\\
  Beijing Institute of Technology\\
  Beijing, 100081, P.R. China}
\email{junhu404@bit.edu.cn}

  \author{Lei Shi}\address{Academy of Mathematics and Systems Science\\
	Chinese Academy of Sciences, Beijing 100190\\
	P.R.China}
\address{Max-Planck-Institut f\"ur Mathematik\\
	Vivatsgasse 7, 53111 Bonn\\
	Germany}
\email{leishi202406@163.com}

\numberwithin{equation}{section}
\newtheorem{prop}[equation]{Proposition}
\newtheorem{thm}[equation]{Theorem}

\newtheorem{cor}[equation]{Corollary}
\newtheorem{conj}[equation]{Conjecture}

\newtheorem{lem}[equation]{Lemma}

\theoremstyle{definition}
\newtheorem{dfn}[equation]{Definition}
\theoremstyle{remark}
\newtheorem{rem}[equation]{Remark}

\begin{document}

\begin{abstract}
	In this paper, we construct some integral bases for the cocenter of the non-degenerate cyclotomic Hecke algebra $\HH_{n,K}$ of type $G(r,1,n)$ by generalizing Geck and Pfeiffer's work on the cocenters of the Iwahori-Hecke algebras associated to finite Weyl groups. We show that the dimensions of both the cocenter and the center of the non-degenerate cyclotomic Hecke algebra $\HH_{n,K}$ are independent of the characteristic of the ground field $K$, its Hecke parameter and cyclotomic parameters. As applications, we verify Chavli-Pfeiffer's conjecture on the polynomial coefficient $g_{w,C}$ (\cite[Conjecture 3.7]{CP}) for the complex reflection group of type $G(r,1,n)$ and also show that both the cocenters and the centers of certain cyclotomic KLR algebras of affine type $A$ are independent of the characteristic of the ground field.
\end{abstract}

\maketitle
\setcounter{tocdepth}{1}
\tableofcontents

\section{Introduction}

Let $r,n\in\Z_{\geq 1}$. The wreath product $(\Z/r\Z)\wr\Sym_n$ of the cyclic group $\Z/r\Z$ with the symmetric group $\Sym_n$ is called the complex reflection group of type $G(r,1,n)$. It can be realized as the group of all monomial matrices of size $n$ whose nonzero entries are $r$th roots of unity.

\begin{dfn}\label{crg} The complex reflection group $W_n$ of type $G(r,1,n)$ is isomorphic to the group presented by the generators $S:=\{t,s_1,\cdots, s_{n-1}\}$ and the following relations:$$\begin{aligned}
		& t^r=s_i^2=1,\, \forall\, 1\leq i\leq n-1;\\
		& ts_1ts_1=s_1ts_1t,\quad ts_i=s_it,\,\,\forall\,2\leq i\leq n-1; \\
		& s_is_{i+1}s_i=s_{i+1}s_is_{i+1},\,\,\forall\, 1\leq i<n-1;\\
		& s_is_j=s_js_i,\, \forall\,1\leq i<j-1<n-1.\end{aligned} $$
\end{dfn}

If $r=1$, then $W_n$ coincides with the symmetric group $\Sym_n$ on $\{1,2,\cdots,n\}$ with standard Coxeter generators $\{s_i=(i,i+1)|i=1,2,\cdots,n-1\}$. If $r=2$, then $W_n$ coincides with the Weyl group of type $B_n$ with standard Coxeter generators $\{t, s_i|i=1,2,\cdots,n-1\}$.

Given $w\in W_n$, a word $x_1\cdots x_k$ on $S=\{t,s_1,\cdots,s_{n-1}\}$ is called an expression of $w$ if $x_i\in S, \forall\,1\leq i\leq k$, and $w=x_1\cdots x_k$. If $x_1\cdots x_k$ is an expression of $w$ with $k$ minimal, then we call it a reduced expression of $w$ and say $w$ has length $k$. We denote $\ell(w)=k$. For any $w,\,w'\in W_n$, we write $w \overset{s}{\rightarrow} w'$ if $w'=sws^{-1}$ for some $s\in S$, $\ell(w')\leq \ell(w)$ and \begin{equation}\label{2possibi0}
	\text{either $\ell(sw)<\ell(w)$ or $\ell(ws^{-1})<\ell(w)$.}
\end{equation}
If $w=w_1,\,w_2,\cdots,\,w_m=w'\in W_n$ such that for any $1\leq i<m$, $w_{i} \overset{x_i}{\rightarrow} w_{i+1}$ for some $x_i\in S$, then we write $w \overset{(x_1,\cdots,x_{m-1})}{\longrightarrow} w'$ or $w\rightarrow w'$.

The following theorem, which generalizes Geck and Pfeiffer's work \cite{GP} on the minimal length elements in each conjugacy class of Weyl groups to the complex reflection group $W_n$, is the first main result of this paper.

\begin{thm}\label{mainthm1} For any conjugacy class $C$ of $W_n$ and any $w\in C$, there exists an element $w'\in C_{\min}$, such that $w\rightarrow w'$, where $C_{\min}$ is the set of minimal length elements in $C$.
\end{thm}

Note that here we use the naive length function $\ell(-)$ for $W_n$ defined by the length of reduced expression in terms of its defining generators. The above generalization of Geck and Pfeiffer's result to the complex reflection group case is quite subtle and nontrivial, mainly due to the fact that the naive length function $\ell(-)$ for $W_n$ does not behave well with respect to the action of $W_n$ on the generalized root system when $W_n$ is not a Weyl group. In particular, Deletion Condition and Exchange Condition do not hold with respect to the naive length function $\ell(-)$ for $W_n$. Note also that $t^{-1}=t^{r-1}\notin S$ whenever $r>2$.

Let $R$ be a commutative ring, $\xi\in R^\times$ and ${\mathbf Q}:=(Q_{1}, \cdots, Q_{r})\in R^r$. The non-degenerate cyclotomic Hecke algebras $\HH_{n,R}$ of type $G(r,1,n)$ were first introduced in \cite[Definition 3.1]{AK}, \cite[Definition 4.1]{BM:cyc} and \cite[before Proposition 3.2]{C} as certain deformations of the group ring $R[W_n]$. They play important roles in the modular representation theory of finite groups of Lie type over fields of non-defining characteristic. By definition, $\HH_{n,R}=\HH_{n,R}(\xi,{\mathbf Q})$ is the unital associative $R$-algebra with generators $T_{0}, T_{1}, \cdots, T_{n-1}$ that are subject to the following relations:
$$
\begin{aligned}
	&\left(T_{0}-Q_{1}\right) \cdots\left(T_{0}-Q_{r}\right)=0, \quad T_{0} T_{1} T_{0} T_{1}=T_{1} T_{0} T_{1} T_{0}; \\
	&\left(T_{i}-\xi \right)\left(T_{i}+1\right)=0, \quad \forall\, 1 \leq i \leq n-1; \\
	&T_{i} T_{j}=T_{j} T_{i},\quad  \forall\, 0 \leq i<j-1<n-1; \\
	&T_{i} T_{i+1} T_{i}=T_{i+1} T_{i} T_{i+1}, \quad \forall\, 1 \leq i<n-1.
\end{aligned}
$$
We call $\xi$ and $Q_{1}, \cdots, Q_{r}$ the Hecke parameter and the cyclotomic parameters of $\HH_{n,R}$ respectively. These algebras include the Iwahori-Hecke algebras associated to the Weyl groups of types $A_{n-1}$, $B_n$ as special cases (i.e., $r=1$ and $r=2$ cases).

For any $R$-algebra $A$, we define $\Tr(A):=A/[A,A]$, and call it the cocenter of $A$, where $[A,A]$ denotes the $R$-submodule of $A$ spanned by $ab-ba$ for all $a,b\in A$. In this paper, we are mainly interested in the cocenter of the cyclotomic Hecke algebra $\HH_{n,R}$ over an arbitrary commutative domain $R$.

Note that the Matsumoto theory for finite Coxeter groups is not applicable to $W_n$ anymore, the product $T_{x_1}\cdots T_{x_k}$ usually does depend on the choice of the reduced expression $x_1\cdots x_k$ of $w\in W_n$ instead of only on $w$. We do not have a uniquely well-defined standard basis for $\HH_{n,R}$ as in the finite Coxeter groups case. Let ${\rm{Cl}}(W_n)$ be the set of conjugacy classes of $W_n$. The following theorem is the second main result of this paper. It gives some integral bases for the cocenter $\Tr(\HH_{n,R})$ of the cyclotomic Hecke algebra $\HH_{n,R}$ and shows that the center is stable under base change.

\begin{thm}\label{mainthm2} Let $R$ be a commutative domain and $\xi, Q_1,\cdots,Q_r\in R^\times$.
	
	1) For each conjugacy class $C$ of $W_n$, we arbitrarily choose an element $w_C\in C_{\min}$ and fix a reduced expression $x_1\cdots x_k$ of $w_{C}$, and define
	$T_{w_{C}}:=T_{x_1}\cdots T_{x_k}$. Then the following set \begin{equation}\label{bas0}
		\bigl\{T_{w_{C}}+[\HH_{n,R},\HH_{n,R}]\bigm|C\in{\rm{Cl}}(W_n)\bigr\}
	\end{equation}
	forms an $R$-basis of the cocenter $\Tr(\HH_{n,R})$. In particular, the cocenter $\Tr(\HH_{n,R})$ is a free $R$-module of rank $|\mathscr{P}_{r,n}|$, where $\mathscr{P}_{r,n}$ is the set of $r$-partitions of $n$.
	
	2) The center $Z(\HH_{n,R})$ is a free $R$-module of rank $|\mathscr{P}_{r,n}|$. Moreover, for any commutative domain $R'$ which is an $R$-algebra, the following canonical map $$
	R'\otimes_{R}Z(\HH_{n,R})\rightarrow Z(\HH_{n,R'})
	$$
	is an $R'$-module isomorphism.

In particular, if $R=K$ is a field, then the dimensions of both the cocenter $\Tr(\HH_{n,K})$ and the center $Z(\HH_{n,K})$ are independent of the characteristic of the ground field $K$, and the Hecke parameters and cyclotomic parameters of $\HH_{n,K}$.
\end{thm}

Let us briefly explain how we prove Theorem \ref{mainthm2}. Adapting a similar argument in Theorem \ref{mainthm1}, we show in Theorem \ref{span} that, over an arbitrary commutative unital ring $R$, $\{T_{w_C}|C\in {\rm{Cl}}(W_n)\}$ gives a $R$-spanning set for the cocenter $\Tr(\HH_{n,R})$ of $\HH_{n,R}$. This gives an upper bound for the dimension of the cocenter $\Tr(\HH_{n,R})$ when $R=K$ is a field. Then we use seminormal basis theory for the semisimple cyclotomic Hecke algebras $\HH_{n,\K}$ and the symmetric structure of $\HH_{n,R}$ to show that this upper bound is also the lower bound of the dimension of the center $Z(\HH_{n,K})$ and hence the dimension of the cocenter $\Tr(\HH_{n,K})$. The coincidence of the upper bound and the lower bound forces Theorem \ref{mainthm2} holds. Note that Brundan \cite{Brundan:degenCentre} has proved that the dimension of the center of the degenerate cyclotomic Hecke algebra of type $G(r,1,n)$ is independent of the characteristic of the ground field and its cyclotomic parameters by explicitly constructing an integral basis. Our argument applies equally well to prove a version of Theorem \ref{mainthm2} in the degenerate setting.

The content of the paper is organised as follows. In Section 2 we introduce some basic notions and fix some notations which will be used in later sections. We recall some preliminary known results on the non-degenerate cyclotomic Hecke algebras of type $G(r,1,n)$. In Section 3 we give a proof of our first main result Theorem \ref{mainthm1}. The whole Section 3 involves only complex reflection group theoretic discussion, but the main result will be used in the proof of Theorem \ref{mainthm2}. In Section 4 we give the proof of our second main result Theorem \ref{mainthm2}.  In Section 5 we give two applications of our main results in this paper.
The first application (Proposition \ref{apply2}) verifies Chavli-Pfeiffer's conjecture on the polynomial coefficient $g_{w,C}$ (\cite[Conjecture 3.7]{CP}) for the complex reflection group of type $G(r,1,n)$, while the second application (Theorem \ref{application2}) shows that both the cocenters and the centers of certain cyclotomic KLR algebras of affine type $A$ are independent of the characteristic of the ground field.

\bigskip
\centerline{Acknowledgements}
\bigskip

The research was supported by the National Natural Science Foundation of China (No. 12431002). The second author is partially supported by the Postdoctoral Fellowship Program of CPSF under Grant Number GZB20250717.
\bigskip

\section{Preliminary}

Let $R$ be a commutative (unital) ring. We use $R^\times$ to denote the set of units in $R$. Let $\HH_{n,R}$ be the non-degenerate cyclotomic Hecke algebra of type $G(r,1,n)$ (defined over $R$) with Hecke parameter $\xi\in R^\times$ and cyclotomic parameters $Q_1,\cdots,Q_r\in R$.

The {\bf Jucys-Murphy elements} of $\HH_{n,R}$ are defined as follows:
\begin{equation}
\mL_{m}:=\xi^{1-m} T_{m-1} \cdots T_{1} T_{0} T_{1} \cdots T_{m-1},\quad m=1,2, \cdots, n.
\end{equation}
These elements commute with each other and any symmetric polynomial in $\mL_1,\cdots,\mL_n$ is central in $\HH_{n,R}$.

\begin{lem}\text{\rm (\cite[Theorem 3.10]{AK})}\label{stdBasis}
The elements in the following set
\begin{equation}\label{basis1}
\bigl\{\mL_1^{c_1}\cdots \mL_n^{c_n}T_w\bigm|w\in\Sym_n, 0\leq c_i<r,\forall\,1\leq i\leq n\bigr\}
\end{equation}
give an $R$-basis of $\HH_{n,R}$.
\end{lem}

\begin{dfn}\label{ssf} For any $w\in \Sym_n$ and integers $0\leq c_1,c_2,\cdots,c_n<r$, we define
$$
\tau_{R}(\mL_1^{c_1}\cdots \mL_n^{c_n}T_w):=\begin{cases}
1, &\text{if $w=1$ and $c_1=\cdots=c_n=0$;}\\
0, &\text{otherwise.}
\end{cases}
$$
We extend $\tau_{R}$ linearly to an $R$-linear function on $\HH_{n,R}$.
\end{dfn}

Let $A$ be an $R$-algebra which is a free $R$-module of finite rank. Recall that $A$ is called a symmetric $R$-algebra if there is an $R$-linear function $\tau: A\rightarrow R$ such that  $\tau(hh')=\tau(h'h), \forall\,h,h'\in A$ and $\tau$ is non-degenerate (i.e., the morphism $\hat{\tau}: A \to \Hom_R(A,R), a\mapsto (a'\mapsto \tau (a'a))$ is an $R$-module isomorphism), see \cite[Definition 7.1.1]{GP}. In this case, $\tau$ is called a symmetrizing form on $A$. It is clear that an $R$-linear function $\tau: A\rightarrow R$ is non-degenerate if and only if there is a pair of $R$-bases $\mathcal{B},\mathcal{B}'$ of $A$ such that the determinant of the matrix $(\tau(bb'))_{b\in \mathcal{B},b'\in\mathcal{B}'}$ is a unit in $R$. If $A$ is a symmetric algebra over $R$, then it follows from \cite[Lemma 7.1.7]{GP2} that there is an $R$-module isomorphism: \begin{equation}\label{ztr}
Z(A)\cong(\Tr(A))^*:=\Hom_R(\Tr(A),R).
\end{equation}
Note that, in general, we do not know whether $\Tr(A)$ is isomorphic to $(Z(A))^*$ or not when $R$ is not a field because $\Tr(A)$ may not be a free $R$-module.

\begin{lem}\text{\rm (\cite{MM})} Suppose that  $Q_1,\cdots,Q_r\in R^\times$. Then $\tau_{R}$ is a symmetrizing form on $\HH_{n,R}$ which makes $\HH_{n,R}$ into a symmetric algebra over $R$.
\end{lem}
Henceforth, we shall call $\tau_{R}$ {\bf the standard symmetrizing form} on $\HH_{n,R}$.

\begin{lem}[\text{\rm \cite[Main Theorem]{A1}}]\label{ss} Let $R=K$ be a field. The  cyclotomic Hecke algebra $\HH_{n,K}$ is semisimple if and only if $$
\Bigl(\prod_{k=1}^{n}(1+\xi+\cdots+\xi^{k-1})\Bigr)\Bigl(\prod_{\substack{1\leq l<l'\leq r\\ -n<k<n}}\bigl(\xi^kQ_l-Q_{l'}\bigr)\Bigr)\in K^\times .
$$
In that case, it is split semisimple.
\end{lem}

Let $d\in\N$. A composition of $d>0$ is a finite sequence $\rho=(\rho_1,\rho_2,\cdots,\rho_k)$ of positive integers which sums to $d$, we write $|\rho| = \sum_{j=1}^k\rho_j=d,\,\ell(\rho)=k$, and call $\ell(\rho)$ the length of $\rho$. By convention, we understand $\emptyset$ as a composition of $0$. An $r$-composition of $d$ is an ordered $r$-tuple $\blam=(\lam^{(1)},\cdots,\lam^{(r)})$ of compositions $\lam^{(k)}$ such that $\sum_{k=1}^{r}|\lam^{(k)}|=d$. A partition of $d$ is a composition $\lambda=(\lambda_1,\lambda_2,\cdots)$ of $d$ such that $\lam_1\geq\lam_2\geq\cdots$. We use $\mathcal{P}_d$ to denote the set of partitions of $d$. An $r$-partition of $d$ is an $r$-composition $\bm{\lambda}=\left(\lambda^{(1)}, \cdots, \lambda^{(r)}\right)$ of $d$ such that each $\lam^{(k)}$ is a partition.
Given a composition $\lambda = (\lambda_1 ,\lambda_2,\cdots )$ of $d$, we define its conjugate $\lambda' = (\lambda'_1 , \lambda'_2 ,\cdots)$ by $\lambda'_k = \# \{ j\geq 1 \mid \lambda_j \geq k \}$, which is a partition of $d$. For any $r$-composition $\blam=(\lam^{(1)},\cdots,\lam^{(r)})$ of $d$, we define its conjugate $\blam':=(\lam^{(r)'},\cdots,\lam^{(1)'})$, which is an $r$-partition of $d$.

We identify the $r$-partition $\bm{\lambda}$ with its Young diagram that is the set of boxes
$$[{\bm{\lambda}}]=\left\{(l, a, c) \mid 1 \leq c \leq \lambda_{a}^{(l)}, 1 \leq l \leq r\right\}.$$
For example, if $\bm{\lambda}=\left((2,1,1),(1,1),(2,1) \right)$ then
\begin{equation*}
[\bm{\lambda}] = \Bigg( \ydiagram{2,1,1},~ \ydiagram{1,1},~ \ydiagram{2,1} \Bigg).
\end{equation*}
The elements of $[\blam]$ are called nodes. Given two nodes $\alpha=(l,a,c), \alpha'=(l',a',c')$, we say that $\alpha'$ is below $\alpha$, or $\alpha$ is above $\alpha'$, if either $l'>l$ or $l'=l$ and $a'>a$.
A node $\alpha$ is called an addable node of an $r$-partition $\blam$ if $[\blam]\cup\{\alpha\}$ is again the Young diagram of an $r$-partition $\bmu$. In this case, we say that $\alpha$ is a removable node of $\bmu$.

We use $\P_{r,n}$ to denote the set of $r$-partitions of $n$. Then $\P_{r,n}$ becomes a poset ordered by dominance ``$\unrhd$'', where $\blam\unrhd\bmu$ if and only if
$$
\sum_{k=1}^{l-1}\left|\lambda^{(k)}\right|+\sum_{j=1}^{i} \lambda_{j}^{(l)} \geq \sum_{k=1}^{l-1}\left|\mu^{(k)}\right|+\sum_{j=1}^{i} \mu_{j}^{(l)},
$$
for any $1 \leq l \leq r$ and any $i \geq 1$. If $\bm{\lambda} \unrhd \bm{\mu}$ and $\bm{\lambda} \neq \bm{\mu}$, then we write $\bm{\lambda} \rhd \bm{\mu}$.

Let $\bm{\lambda}\in\Parts[r,n]$. A $\bm{\lambda}$-tableau is a bijective map $\t: [\bm{\lambda}] \mapsto \{ 1, 2,...,n\}$, for example,

\begin{equation*}
\t = \Bigg( \begin{ytableau}
            1 & 2 \\
            3  \\
            4
            \end{ytableau},~~
            \begin{ytableau}
            5\\
            6
            \end{ytableau},~~
            \begin{ytableau}
            7 & 8 \\
            9
            \end{ytableau} \Bigg)
\end{equation*}
is a $\bm{\lambda}$-tableau, where $\bm{\lambda}:=\left((2,1,1),(1,1),(2,1) \right)$. If $\t$ is a $\bm{\lambda}$-tableau, then we set $\Shape(\t):=\bm{\lambda}$, and we define $\mathfrak{t}'\in\Std(\blam')$ by
$\t'(l,a,c):=\t(r+1-\ell,c,a)$ and call $\t'$ the conjugate of $\t$.

A $\bm{\lambda}$-tableau is standard if its entries increase along each row and each column in each component.
Let $\Std(\bm{\lambda})$ be the set of standard $\blam$-tableaux and $\Std^2(\bm{\lambda}):=\{(\s, \t) \mid \s, \t \in \Std(\bm{\lambda})\}$.
We set $\Std^2(n):=\{(\s,\t)\mid (\s,\t)\in\Std^2(\blam),\blam\in\P_{r,n}\}$.

Let $\blam\in\Parts[r,n], \t\in\Std(\blam)$ and $1\leq m\leq n$. We use $\t_{\downarrow m}$ to denote the subtableau of $\t$ that contains the numbers $\{1, 2,...,m\}$. If $\t$ is a standard $\bm{\lambda}$-tableau then $\Shape(\t_{\downarrow m})$ is an $r$-partition for all $m \geq 0$.
We define $\s \unrhd \t$ if
and only if $$
\Shape( \s\!\downarrow_{m} ) \unrhd \Shape (\t\downarrow_m),\quad\forall\,1\leq m\leq n.
$$
If $\s \unrhd \t$ and $\s \neq \t$, then write $\s \rhd \t$. For any $(\u,\v), (\s,\t)\in\Std^2(n)$, we define $(\u,\v)\unrhd(\s,\t)$ if either $\Shape(\u)=\Shape(\v)\rhd\Shape(\s)=\Shape(\t)$, or $\Shape(\u)=\Shape(\v)=\Shape(\s)=\Shape(\t)$, $\u\unrhd\s$ and $\v\unrhd\t$. If $(\u,\v)\unrhd(\s,\t)$ and $(\u,\v)\neq(\s,\t)$, then we write $(\u,\v)\rhd(\s,\t)$.

Let $\t^{\bm{\lambda}}$ be the standard $\blam$-tableau which has the numbers $1, 2,\cdots,n$ entered in order from left to right along the rows of $\lambda^{(1)}$ and then $\lambda^{(2)}, \cdots, \lambda^{(r)}$. Similarly, let $\t_{\bm{\lambda}}$ be the standard $\blam$-tableau which has the numbers $1, 2, \cdots, n$ entered in order down the columns of $\lambda^{(r)}, \cdots, \lambda^{(1)}$. There is a natural right action of the symmetric group $\Sym_n$ on the set of $\lam$-tableaux. Given a standard $\bm{\lambda}$-tableau $\t$, we define $d(\t), d'(\t)\in \Sym_n$ such that $\t = \t^{\bm{\lambda}} d(\t)$ and $\t_\blam d'(\t)=\t$, and set $w_{\bm{\lambda}}:=d(\t_{\bm{\lambda}})$. For any
$\t\in\Std(\bm{\lambda})$, we have $\t^{\bm{\lambda}} \unrhd \t \unrhd \t_{\bm{\lambda}}$.  The Young subgroup $\Sym_\blam$ is defined
to be the subgroup of $\Sym_n$ consisting of elements which permute numbers in each row of $\tlam$.

Recall that the cyclotomic Hecke algebra $\HH_{n,R}$ is generated by $T_0,T_1,\cdots,T_{n-1}$ with Jucys-Murphy elements $\mL_1,\cdots,\mL_n$.

\begin{dfn}[\text{\rm cf. \cite{DJM}, \cite{Ma}}]\label{mnlam}
Let $\bm{\mu}\in \P_{r,n} $. We define
$$\begin{aligned}
\fm_{\tmu\tmu}&:=\left(\sum_{w \in \Sym_{\mu}} T_{w}\right)\left(\prod_{k=2}^{r} \prod_{m=1}^{\left|\mu^{(1)}\right|+\cdots+\left|\mu^{(k-1)}\right|}\left(\mL_{m}-Q_{k}\right)\right).
\end{aligned}$$
\end{dfn}

Let $\ast$ be the unique anti-involution of $\HH_{n,R}$ which fixes all its defining generators $T_0,T_1,\cdots,T_{n-1}$.

\begin{dfn}[{\cite{DJM}, \cite{Ma}, \cite[(3.3)]{HW}}]\label{nst1}
Let $\bm{\lambda}\in\P_{r,n}$. For any $\s,\t\in \Std(\bm{\lambda})$, we define
$$
\fm_{\s\t}:=\left(T_{d(\s)}\right)^{*} \fm_{\tlam\tlam} T_{d(\t)}.
$$
\end{dfn}

\begin{lem}[\cite{DJM, Ma}]\label{cellular1} The set $\{\fm_{\s \t} \mid \s, \t \in \Std(\bm{\lambda}), \bm{\lambda} \in \P_{r,n}\}$, together with the poset $(\P_{r,n},\unrhd)$ and the anti-involution ``$\ast$'', form a cellular basis of $\HH_{n,R}$ in the sense of \cite{GL}. 
\end{lem}

Sometimes in order to emphasize the ground ring $R$ we shall use the notation $\fm_{\s\t}^R$ instead of $\fm_{\s\t}$.


We now recall some basic results on the semisimple representation theory of the cyclotomic Hecke algebra of type $G(r,1,n)$. Let $x$ be an indeterminate over $R$ and $\K$ be a field which contains both $x$ and $R$. We define
$$\hat{\xi}:=x+\xi\in \K^\times,\quad \hat{Q}_k:=x^{kn}+Q_k,\,\,k=1,2,\cdots,r. $$
Applying (\ref{ss}), we can deduce that the non-degenerate cyclotomic Hecke algebra $\HH_{n,\K}:=\HH_{n,\K}(\hat{\xi};\hat{Q}_1,\cdots,\hat{Q}_r)$ is semisimple.

Let $\blam\in\P_{r,n}$. For any $\gamma=(l,a,b)\in[\blam]$, we define $$
\cont(\gamma):=\hat{Q}_{l}\hat{\xi}^{b-a}\in\K .
$$
For any $\t=(\t^{(1)},\cdots,\t^{(r)}) \in \Std(\blam)$ and $1\leq k\leq n$, if $\t^{-1}(k)=\gamma$ then we define
\begin{equation}\label{content}
\cont(\t)=(\cont(\t^{-1}(1)),\cdots,\cont(\t^{-1}(n))) .
\end{equation}

\begin{lem}\text{\rm (\cite[2.5]{Ma})}\label{sscont} Suppose that $\HH_{n,\K}=\HH_{n,\K}(\hat{\xi};\hat{Q}_1,\cdots,\hat{Q}_r)$ is semisimple. Let $\s\in\Std(\blam), \t\in\Std(\bmu)$, where $\blam,\bmu\in\P_{r,n}$. Then $\s=\t$ if and only if $\cont(\s)=\cont(\t)$.
\end{lem}

For each $1\leq k\leq n$, we define $C(k):=\left\{\cont(\t^{-1}(k)) \mid \t \in \Std(\bm{\lambda}), \bm{\lambda} \in \P_{r, n}\right\}$.

\begin{dfn}\text{\rm (\cite[Definition 2.4]{Ma})}\label{dfn:Ft}
Let $\bm{\lam} \in \P_{r, n}$ and $\t\in\Std(\bm{\lam})$. We define
$$
{F}_{\t}=\prod\limits^n\limits_{k=1}\prod\limits_{\substack{c\in C(k)\\c\neq\cont(\t^{-1}(k))}}\frac{\mL_k-c}{\cont(\t^{-1}(k))-c}.
$$
For any $\bm{\lam} \in \P_{r, n}$ and $\s,\t\in\Std(\bm{\lam})$, we define
$$
\mathfrak{f}_{\s\t}:={F}_\s \fm^{\K}_{\s\t}{F}_\t.
$$
\end{dfn}

\begin{lem}\text{\rm (\cite[2.6, 2.11]{Ma})}\label{fsemi} 1) For any $\s,\t\in\Std(\blam), \u,\v\in\Std(\bmu)$, where $\blam,\bmu\in\P_{r,n}$, we have $$\begin{aligned}
\mathfrak{f}_{\s\t}\mathfrak{f}_{\u\v}=\delta_{\t\u}\gamma_\t \mathfrak{f}_{\s\v},
\end{aligned}
$$
for some $\gamma_\t\in\K^\times$. Moreover, $F_\s=\mathfrak{f}_{\s\s}/\gamma_\s$.

2) For each $\blam\in\P_{r,n}$, ${F}_\blam:=\sum_{\u\in\Std(\blam)}{F}_\u$ is a central primitive idempotent of $\HH_{n,\K}$. Moreover, the set $\{{F}_\bmu|\bmu\in\P_{r,n}\}$ is a complete set of
pairwise orthogonal central primitive idempotents in $\HH_{n,\K}$.
\end{lem}
We shall call $\bigl\{\mathfrak{f}_{\s\t}\bigm|\s,\t\in\Std(\blam),\blam\in\Parts[r,n]\bigr\}$ {\it the seminormal basis} of $\HH_{n,\K}$. The following result was proved in \cite[Theorem 2.19]{Ma}. Here we give a second elementary proof.

For any two $n$-tuples $(a_1,\cdots,a_n), (b_1,\cdots,b_n)\in\K^n$, we define $$
(a_1,\cdots,a_n)\sim (b_1,\cdots,b_n)\Longleftrightarrow (a_1,\cdots,a_n)=\sigma(b_1,\cdots,b_n),\,\,\text{for some $\sigma\in\Sym_n$.}
$$

\begin{lem}\label{sscont2} Let $\blam,\bmu\in\P_{r,n}$. Then $\blam=\bmu$ if and only if
$\cont(\tlam)\sim \cont(\tmu)$.
\end{lem}

\begin{proof} Suppose that $\cont(\tlam)\sim \cont(\tmu)$. By Lemma \ref{ss}, we see that for any $1\leq i\neq j\leq r$, none of the nodes in $[\lam^{(i)}]$ has the same content with a node in $[\lam^{(j)}]$. Thus the assumption $\cont(\tlam)\sim \cont(\tmu)$ implies that for each $1\leq j\leq r$, $\cont(\t^{\lam^{(j)}})\sim \cont(\t^{\mu^{(j)}})$. Now let $1\leq j\leq r$. Lemma  \ref{ss} implies that two nodes in $[\lam^{(j)}]$ have the same contents if and only if they lie in the same diagonal. The same is true for $[\t^{\mu^{(j)}}]$. Note that the lengths of these diagonals uniquely determine the partitions $\lam^{(j)}$ and $\mu^{(j)}$. Thus we can conclude that $\lam^{(j)}=\mu^{(j)}$ for each $1\leq j\leq r$. Hence $\blam=\bmu$.
\end{proof}

\begin{lem}\text{\rm (\cite[Theorem 2.19]{Ma})}\label{Ksymme} For each $\blam\in\P_{r,n}$, $\mathcal{F}_\blam$ is equal to a symmetric $\K$-polynomial in $\mL_1,\cdots,\mL_n$. In particular, the center of $\HH_{n,\K}$ is the set of symmetric $\K$-polynomials in $\mL_1,\cdots,\mL_n$.
\end{lem}

\begin{proof} Note that $\HH_{n,\K}$ is split semisimple. By Lemma \ref{sscont2}, for any $\blam\neq\bmu\in\P_{r,n}$, $$
\cont(\tlam)\not\sim \cont(\tmu).
$$
It follows that there exists an elementary symmetric polynomial $e_{m_{\blam,\bmu}}(X_1,\cdots,X_n)\in\K[X_1,\cdots,X_n]$, where $1\leq m_{\blam,\bmu}\leq n$,  such that $$
e_{m_{\blam,\bmu}}(\cont(\tlam))-e_{m_{\blam,\bmu}}(\cont(\tmu))\in\K^\times.
$$
Now we define a polynomial $g_\blam(X_1,\cdots,X_n)\in\K[X_1,\cdots,X_n]$ as follows: $$
g_\blam(X_1,\cdots,X_n):=\prod_{\substack{\bmu\in\P_{r,n}\\ \bmu\neq\blam}}\frac{e_{m_{\blam,\bmu}}(X_1,\cdots,X_n)-e_{m_{\blam,\bmu}}(\cont(\tmu))}{e_{m_{\blam,\bmu}}(\cont(\tlam))-
e_{m_{\blam,\bmu}}(\cont(\tmu))}.
$$
It is clear that $g_\blam(X_1,\cdots,X_n)$ is a symmetric polynomial in $X_1,\cdots,X_n$. Hence $g_\blam(\mL_1,\cdots,\mL_n)$ is central in $\HH_{n,\K}$. Moreover, by construction and Lemma \ref{fsemi}, $g_\blam(\mL_1,\cdots,\mL_n)$ acts as the identity on the simple module $S_\K^\blam$, and acts as zero on the simple module $S_\K^\bmu$ whenever $\bmu\neq\blam$. Hence we can deduce that $g_\blam(\mL_1,\cdots,\mL_n)=\mathcal{F}_\blam$. Since $\{\mathcal{F}_\blam|\blam\in\P_{r,n}\}$ is a $\K$-basis of the center $Z(\HH_{n,\K})$, we complete the proof of the lemma.
\end{proof}

\bigskip
\section{Minimal length elements in each conjugacy class of $W_n$}

The purpose of this section is to generalize a fundamental result of Geck and Pfeiffer on the minimal length elements in each conjugacy class of finite Weyl groups to the complex reflection group $W_n$ case. The generalization is quite subtle and nontrivial, mainly due to the fact that when $W_n$ is not a Weyl group, it does not have a good length function which behaves well with respect to its action on a suitable generalized root system.

There are actually two versions of length functions for $W_n$: the first one is the naive length function $\ell(-)$ for $W_n$ defined by the length of reduced expression in terms of its defining generators; the second one is the length function defined by the action of $W_n$ on the generalized root system \cite[\S3]{BM}. When $W_n$ is a Weyl group, these two length functions coincide. Bremke and Malle \cite{BM} studied in details the second length function, while we shall use the first naive length function for $W_n$ throughout this paper. Recall that for $w\in W_n$, a word $x_1\cdots x_k$ on $S=\{t,s_1,\cdots,s_{n-1}\}$ is an expression of $w$ if $x_i\in S, \forall\,1\leq i\leq k$, and $w=x_1\cdots x_k$. If $x_1\cdots x_k$ is an expression of $w$ with $k$ minimal, then we call it a reduced expression of $w$ and write $\ell(w):=k$. Note that if $r\in\{1,2\}$, the Matsumoto theory for Weyl groups ensures that the product $T_{x_1}\cdots T_{x_k}$ depends only on $w$ but not on the choice of the reduced expression $x_1\cdots x_k$ of $w$, and thus one can define $T_w:=T_{x_1}\cdots T_{x_k}$ without causing any ambiguity; while if $r>2$, Matsumoto theory is not applicable anymore and thus the product $T_{x_1}\cdots T_{x_k}$ usually does depend on the choice of the reduced expression $x_1\cdots x_k$ instead of only on $w$.

\subsection{Normal forms and Double coset decomposition}

Recall the presentation for the complex reflection group $W_n$ given in Definition \ref{crg}, where the last four relations are usually called braid relations. By definition, we have $(s_1ts_1)t=t(s_1ts_1)$. It follows that for any $a,b\in\N$, \begin{equation}\label{wbrela}
s_1t^as_1t^b=(s_1ts_1)^at^b=t^b(s_1ts_1)^a=t^bs_1t^as_1 .
\end{equation}

\begin{dfn} For each $0\leq k\leq n-1,\,a\in \N, l\in\Z_{\geq 1}$, we define
$$t_{k,a}:=\begin{cases}s_ks_{k-1}\cdots s_1 t^a, &\text{if $a\neq 0$;}\\
1, &\text{if $a=0$,}
\end{cases}
$$ and $$s'_{k,l}:=s_ks_{k-1}\cdots s_1 t^ls_1\cdots s_{k-1}s_k.$$
By convention, $t_{0,a}$ is understood as $t^a$, $s'_{0,l}$ is understood as $t^l$.
\end{dfn}

\begin{dfn} For any two expressions $x_{i_1}\cdots x_{i_k}$ and $x_{j_1}\cdots x_{j_l}$ of $w\in W_n$, where $x_{i_a},\, x_{j_b}\in S, \forall\,a,b$, we say they are weakly braid-equivalent if one can use braid relations together with the relation (\ref{wbrela}) to transform one to another.
\end{dfn}

Since braid relations and the relation (\ref{wbrela}) keep the length invariant, it is clear that if two expressions are weakly braid-equivalent, then one of them is reduced if and only if the other one is reduced.

\begin{lem}[\text{\cite[Lemma 1.5]{BM}}]\label{BM2}
Any reduced expression of $w\in W_n$ is uniquely weakly braid-equivalent to a word of the form \begin{equation}\label{canonical}
t_{0,a_0}\cdots t_{n-1,a_{n-1}}v, \,\,\,\text {where} \,\,0\leq a_i\leq r-1,\,\, v\in \Sym_n\,\, \text{reduced}.
\end{equation} Moreover, the words of the shape \eqref{canonical} are all reduced and form a system of representatives of all elements of $W_n$.
\end{lem}
We call (\ref{canonical}) the \textbf{BM normal forms} of elements in $W_n$. By convention, a consecutive sequence of the form $s_as_{a+1}\cdots s_k$ or $s_{a}s_{a-1}\cdots s_k$ is understood as identity whenever $k=0$.

\begin{lem}\text{(\cite[(3.14),(3.15)]{BM})}\label{addone} Let $w\in W_n$ and $s\in S=\{t,s_1,\cdots,s_{n-1}\}$. Then $$
\ell(ws)\leq\ell(w)+1,\quad \ell(sw)\leq\ell(w)+1 .
$$
\end{lem}

\begin{prop}\label{canonicalform}
Any reduced expression of $w\in W_n$ is uniquely weakly braid-equivalent to a reduced word of one of the following forms: $$\begin{aligned}\label{double coset}
(1)\,\,&t_{0,a_0}\cdots t_{n-2,a_{n-2}}\sigma s_{n-1}\cdots s_k,\,\,0\leq k\leq n-1,\\
(2)\,\,&t_{0,a_0}\cdots t_{n-2,a_{n-2}}\sigma s_{n-1}\cdots s_1 t^l s_1\cdots s_k,\,\,0\leq k\leq n-2,\,1\leq l\leq r-1,\\
(3)\,\,&t_{0,a_0}\cdots t_{n-2,a_{n-2}}\sigma s'_{n-1,l},\,\,1\leq l\leq r-1,
\end{aligned}$$
where in each expression, $\sigma\in \Sym_{n-1}$ is a reduced expression. Moreover, these words (1), (2) and (3) form a system of representatives of all elements of $W_n$.
\end{prop}
Later in Corollary \ref{cancel1} we shall see that (1), (2), (3) give rise to a nice $(W_{n-1},W_{n-1})$-double coset decomposition of all elements in $W_n$. Therefore, we shall refer the above three kinds of words (1), (2), (3) as \textbf{double coset normal forms} (or \textbf{DC normal forms} for short)  of elements in $W_n$.

\begin{proof}
By Lemma \ref{BM2}, each reduced expression of $x\in W_n$ is uniquely weakly braid-equivalent to a word of the form \eqref{canonical}.

\medskip
{\it Case 1.} $a_{n-1}=0$. Then the expression (\ref{canonical}) is of the form $$t_{0,a_0}\cdots t_{n-2,a_{n-2}}v,
$$ where  $v\in \Sym_{n}$ is a reduced expression. But we have the canonical right coset decomposition $$v=\sigma s_{n-1}\cdots s_k,
$$ where $\sigma\in\Sym_{n-1}$ and $0\leq k \leq n-1$. Hence, it is weakly braid-equivalent to one of the elements in (1).

\medskip
{\it Case 2.} $a_{n-1}\neq 0$. We also have the canonical right coset decomposition of $v$:$$ v=\sigma's_1\cdots s_k,
$$
where $\sigma'\in \Sym_{\{2,3,\cdots n\}}$ and $0\leq k \leq n-1$. Using the braid relations for $W_n$ we see that (\ref{canonical}) is weakly braid-equivalent to the form of $$t_{0,a_0}\cdots t_{n-2,a_{n-2}}\sigma t_{n-1, a_{n-1}}s_1\cdots s_k,
$$ where $\sigma\in \Sym_{n-1}$ and $0\leq k \leq n-1$. This is exactly an element of either the form (2) or the form (3) in this proposition.

Finally, one can check that the numbers of the expressions (1), (2), (3) above is exactly $|W_n|$. It follows that these elements are distinct and hence the last statement of the proposition holds.
\end{proof}

\begin{dfn} For each $n\in\Z_{\geq 1}$, we define $$\D_n=\{1,\,s_{n-1},\,s'_{n-1,1},\cdots, s'_{n-1,r-1}\}.
$$
By convention, $\D_1:=\{1,t,t^2,\cdots,t^{r-1}\}$.
\end{dfn}

\begin{cor}\label{cancel1}
For any $w\in W_n$, there is a unique element $d_n\in\D_n$, such that Proposition \ref{double coset} gives the following decomposition:
\begin{equation}\label{xadnb} w=ad_nb,
\end{equation} with the property that $\ell(w)=\ell(a)+\ell(d_n)+\ell(b)$ and $a,b\in W_{n-1}$. Moreover, if $b$ ends with $s\in S\setminus\{s_{n-1}\}$, then $$ sws^{-1}=(sa)d_n(bs^{-1})
$$ can become a DC normal form (\ref{canonicalform}) if we rewrite $sa$ to be the form of \eqref{canonical}. Moreover, $\ell(sws^{-1})\leq \ell(w)$.
\end{cor}

\begin{proof} The first statement is clear. Let's consider the second statement. Suppose $b$ ends with $s\in S\setminus\{s_{n-1}\}$, we can rewrite $sa$ to be the form of \eqref{canonical}.

\medskip
{\it Case 1.} $s=t$. Then the double coset decomposition (\ref{xadnb}) must be a DC normal form (2) in Proposition  \ref{canonicalform} (with $k=0$, $d_n=s_{n-1}$ and $a=t_{0,a_0}\cdots t_{n-2,a_{n-2}}\sigma$). That is,
$$t_{0,a_0}\cdots t_{n-2,a_{n-2}}\sigma s_{n-1}\cdots s_1 t^l,
$$where $1\leq l\leq r-1$ and $\sigma\in \Sym_{n-1}$ is a reduced expression. Then $$
twt^{-1}=tas_{n-1}s_{n-2}\cdots s_1 t^{l-1}=t_{0,a_0+1}\cdots t_{n-2,a_{n-2}}\sigma s_{n-1}\cdots s_1 t^{l-1}$$ and $\ell(ta)=\ell(a)+1$ if $a_0<r-1$; while $\ell(ta)=\ell(a)-(r-1)$ when $a_0=r-1$. This proves $\ell(sws^{-1})\leq \ell(w)$ in this case.

\medskip
{\it Case 2.} $s=s_i$, where $1\leq i<n-1$. Then by Lemma \ref{addone}, $\ell(sa)\leq \ell(a)+1$.

Hence in both two cases, we have$$\begin{aligned}
\ell(sws^{-1})=\ell(sad_nbs^{-1})&=\ell(sa)+\ell(d_n)+\ell(bs^{-1})\leq \ell(a)+1+\ell(d_n)+\ell(bs^{-1})\\
&=\ell(a)+\ell(d_n)+\ell(b)=\ell(ad_nb).\end{aligned}
$$
\end{proof}

\begin{cor}\label{cancel12}
For any $d_n\in\D_n$ and $w\in W_{n-1}$, we have $\ell(wd_n)=\ell(w)+\ell(d_n)$.
\end{cor}

\begin{proof} We express $w$ in the form (\ref{canonicalform}). Then the corollary follows from Corollary \ref{cancel1}.
\end{proof}

\subsection{Some minimal length elements in conjugacy class}

Let $\lambda=(\lambda_1,\cdots,\lambda_k)$ be a composition of $n$. We set $r_1:=0$, $r_{k+1}:=n$, and $$
r_i:=\lam_1+\lam_2+\cdots+\lam_{i-1},\quad\forall\,2\leq i\leq k .
$$
Let $J:=\{0,1,\cdots,r-1\}$ and $\epsilon=(\epsilon_1,\cdots,\epsilon_k)\in J^k$. For each $1\leq i\leq k$, we define \begin{equation}\label{wlam}
w_{\lambda,\epsilon,i}:=\begin{cases}s'_{r_i,\epsilon_i}s_{r_i+1}s_{r_i+2}\cdots s_{r_{i+1}-1}, &\text{if $\epsilon_i\neq 0$;}\\
s_{r_i+1}s_{r_i+2}\cdots s_{r_{i+1}-1}, &\text{if $\epsilon_i=0$,}
\end{cases},\quad w_{\lambda,\epsilon}=\prod_{i=1}^k w_{\lambda,\epsilon,i}.
\end{equation}

Recall that for each $m\in\N$, $\mathcal{P}_m$ denotes the set of partitions of $m$.

\begin{dfn} A composition $\lam=(\lam_1,\cdots,\lam_k)$ of $n$ is called an opposite partition if $\lam_1\leq\lam_2\leq\cdots\leq\lam_k$. We use $\OP_{m,-}$ to denote the set of opposite partitions of $m$. Given $\lam=(\lam_1,\cdots,\lam_k)\in\OP_{m,-}$, we color each row $i$ of $\lam$ with an integer $c(i)\in\{1,\cdots,r-1\}$ such that $c(i)\geq c(i+1)$ whenever $\lam_i=\lam_{i+1}$.
\end{dfn}

\begin{dfn}\label{colorsemi} If $\lam$ is an opposite partition of $m$ with a color data $\{c(i)|1\leq i\leq\ell(\lam)\}$, $\mu$ is a composition of $n-m$, then we call the bicomposition $(\lam,\mu)$ a colored semi-bicomposition of $n$. We use $\CC_{n}$ to denote the set of colored semi-bicompositions of $n$. If $(\lam,\mu)\in\CC_n$ and $\mu$ is a partition, then we say $(\lam,\mu)$ is a colored semi-bipartition. We use $\CP_n$ to denote the set of colored semi-bipartitions of $n$.
\end{dfn}

For each colored semi-bicomposition ${\alpha}=(\lam,\mu)\in\CC_n$, where $\lam=(\lam_1,\cdots,\lam_k)$ and  $\mu=(\mu_1,\cdots,\mu_l)$,  we associate it with a composition $\overline{{\alpha}}:=(\lam_1,\cdots,\lam_k,\mu_1,\cdots,\mu_l)$ of $n$ and a sequence $\epsilon=(c(1),\cdots,c(k),\underbrace{0,\cdots,0}_{\text{$l$ copies}})\in J^{k+l}$. We define \begin{equation}\label{walpha}
w_{\alpha}:=w_{\overline{\alpha},\epsilon}.
\end{equation}

The following combinatorial result follows directly from the definition of colored semi-bipartitions.

\begin{lem}\label{bije12} There is bijection $\theta$ from the set $\CP_n$ onto the set $\mathscr{P}_{r,n}$ of $r$-partitions of $n$ such that \begin{enumerate}
\item the $1$-st component of $\theta(\lam,\mu)$ is $\mu$; and
\item for each $2\leq i\leq r$, the $i$-th component of $\theta(\lam,\mu)$ is the unique partition obtained by reordering the order of all the rows of $\lam$ colored by $i-1$.
\end{enumerate}
\end{lem}

We set $$\begin{aligned}
\Sigma_n:&=\bigl\{(d_1,\cdots,d_n)\bigm|d_i\in\mathcal{D}_i, \forall\,1\leq i\leq n\bigr\},\\
\mathcal{C}_n:&=\Bigl\{(\lam,\epsilon)\Bigm|\begin{matrix}\text{$\lam=(\lam_1,\cdots,\lam_k)$ is a composition of $n$,}\\
\text{$\epsilon=(\epsilon_1,\cdots,\epsilon_k)\in J^k$.}
\end{matrix}\Bigr\}.
\end{aligned}
$$

\begin{lem}\label{bije1} With the notations as above,  there is a natural bijection $\theta_n$ from the set $\Sigma_n$ onto the set $\mathcal{C}_n$.
\end{lem}

\begin{proof} We construct inductively a bijection $\theta_n$ from the set $\Sigma_n$ onto the set $\mathcal{C}_n$ as follows.  For any $1<m<n$, if $$
d_{m+1}=s_m, $$
then we say that $\{d_m,d_{m+1}\}$ are consecutive, otherwise we say $\{d_m,d_{m+1}\}$ are not consecutive. For example,  $\{d_1,d_{2}\}$ are consecutive if and only if $(d_1,d_{2})\in\{(t^a,s_{1})|0\leq a\leq r-1\}.$

If $n=1$, then we define $\theta_1(d_1)=((1),a)$, where $0\leq a\leq r-1$ satisfying $d_1=t^a$, $(1)$ denotes the one box composition of $1$. In general, assume that for each $1\leq m\leq n-1$, the bijection map $\theta_m$ is already constructed. Suppose that $d_{n-1}, d_n$ are not consecutive. If $d_n=1$ (resp., $d_n=s'_{n-1,a}$ for some $1\leq a\leq r-1$), then we define $\lam(n)$ to be the composition of $n$ which is obtained by adding a one box row to the bottom of $\lam(n-1)$ and define $\epsilon(n)$ to be tuple obtained by adding one more component with entry $0$ (resp., $a$) to the right end of $\epsilon(n-1)$;

Suppose that $d_{n-1}, d_n$ are consecutive. Let $m$ be the minimal integer such that for any $0\leq i\leq n-m-1$, $d_{m+i}, d_{m+i+1}$ are consecutive. In particular, $d_{m-1},d_m$ are not consecutive. If $d_m=1$, then we define $\lam(n)$ to be the composition of $n$ which is obtained by adding an $n-m+1$ boxes row to the bottom of $\lam(m-1)$ and define $\epsilon(n)$ to be tuple obtained by adding one more component with entry $0$ to the right end of $\epsilon(m-1)$; If $d_m=s'_{m-1,a}$ for some $1\leq a\leq r-1$, then we define $\lam(n)$ to be the composition of $n$ which is obtained by adding an $n-m+1$ boxes row to the bottom of $\lam(m-1)$ and define $\epsilon(n)$ to be tuple obtained by adding one more component with entry $a$ to the right end of $\epsilon(m-1)$. As a result, we get a composition $\lam=(\lam_1,\cdots,\lam_k)$ of $n$ and a sequence $\epsilon=(\epsilon_1,\cdots,\epsilon_k)\in J^k$ which satisfies $d_1\cdots d_n= w_{\lambda,\epsilon}$. In other words, we have defined the map $\theta_n$. Conversely, as any element $w_{\lam,\epsilon}$ can be uniquely decomposed as $d_1\cdots d_n$ with $d_i\in\D_i$ for each $i$, we see there is a natural map $\theta'_n$ from the set $\mathcal{C}_n$ to the set $\Sigma_n$. It is easy to check that $\theta'_n\circ\theta_n=\id$ and $\theta_n\circ\theta'_n=\id$. Hence $\theta_n$ is a bijection.
\end{proof}

\begin{dfn}
Given $w,\,w'\in W_n$ and $s\in S$, we write $w \overset{s}{\rightarrow} w'$ if $w'=sws^{-1}$, $\ell(w')\leq \ell(w)$ and \begin{equation}\label{2possibi}
\text{either $\ell(sw)<\ell(w)$ or $\ell(ws^{-1})<\ell(w)$.}
\end{equation}
If $w=w_1,\,w_2,\cdots,\,w_m=w'$ is a sequence of elements such that for each $1\leq i<m$, $w_{i} \overset{x_i}{\rightarrow} w_{i+1}$ for some $x_i\in S$, we write $w \overset{(x_1,\cdots,x_{m-1})}{\longrightarrow} w'$ or $w\rightarrow w'$.
\end{dfn}

Note that if $s\in\{s_1,\cdots,s_{n-1}\}$, then using Lemma \ref{addone} we can deduce that the condition (\ref{2possibi}) implies that $\ell(w')=\ell(sws)\leq \ell(w)$.

\begin{prop}\label{reduce to composition}
For each $w\in W_n$, there exists a composition $\lambda=(\lambda_1,\cdots,\lambda_k)$ of $n$, a sequence $\epsilon\in J^k$ and a sequence $x_1,\cdots,x_m$ of defining generators in $W_{n-1}$, such that $w \overset{(x_1,\cdots,x_m)}{\longrightarrow} w_{\lambda,\epsilon}$.
\end{prop}

\begin{proof} We consider the DC normal form of $w$ as given in Proposition \ref{canonicalform}. We can write $w=ad_nb$, where $$\begin{aligned}
a&=t_{0,a_0}\cdots t_{n-2,a_{n-2}}\sigma,\quad \sigma\in\Sym_{n-1}, 0\leq a_i\leq r-1,\forall\,0\leq i\leq n-2 ,\\
b&=\begin{cases} \begin{matrix}\text{$s_{n-2}\cdots s_1t^l s_1\cdots s_k$}\\
\text{or $s_{n-2}s_{n-3}\cdots s_{k'}$},\end{matrix} &\text{if $d_n=s_{n-1}$;}\\
1, &\text{if $d_n=1$ or $d_n=s'_{n-1,l}$ for some $1\leq l\leq r-1$,}
\end{cases}
\end{aligned}
$$
where $1\leq k'\leq n-1$, $0\leq k\leq n-2$.

Now applying Corollary \ref{cancel1}, we shows that $w\overset{\sigma_n}{\rightarrow}w'd_n$, where $$\sigma_n=(x_{n1},\cdots,x_{nl_n}), \,\, x_{nj}\in\{t,s_1,\cdots,s_{n-2}\},\,\forall\,1\leq j\leq l_n,\,\, w'\in W_{n-1}.$$  Applying Corollary \ref{cancel1} to $w'$, we can write $$
w'=a'd_{n-1}b',$$
where $a',b'\in W_{n-2}$. In particular, both $a', b'$ commute with $d_n$. Applying Corollaries \ref{cancel1} and \ref{cancel12} again, we can write $w'd_n\overset{\sigma_{n-1}}{\rightarrow}w''d_{n-1}d_n$, where $\sigma_{n-1}$ is a sequence of standard generators in $W_{n-2}$, $w''\in W_{n-2}$.  Repeating this procedure, eventually we arrive that
$$w \overset{\sigma_n\sigma_{n-1}\cdots \sigma_1}{\rightarrow} d_1\cdots d_n,
$$ where $d_1\in\{1,t,t^2,\cdots,t^{r-1}\}$.  Applying Lemma \ref{bije1}, we see that $d_1\cdots d_n= w_{\lambda,\epsilon}$ for some composition $\lambda=(\lambda_1,\cdots,\lambda_k)$ of $n$ and a sequence $\epsilon=(\epsilon_1,\cdots,\epsilon_k)\in J^k$. We are done.
\end{proof}

\begin{lem}\label{symmetric reduce} Let $j\in\Z_{\geq 0}$ and $w_j\in W_j$. Suppose $$\begin{aligned}
& w=(s_{j+1}\cdots s_{j+m})(s_{j+m+2}\cdots s_{j+m+k+1}),\\
& u=(s_{j+1}\cdots s_{j+k})(s_{j+k+2}\cdots s_{j+k+m+1}).
\end{aligned} $$
Then \begin{enumerate}
\item there exists $y\in \Sym_{\{j+1,\cdots,j+m+k+2\}}$ such that $y^{-1}wy=u$ and $\ell(wy)=\ell(w)+\ell(y)$;
\item Moreover, $y^{-1}w_jwy=w_ju$, $\ell(w_jw)=\ell(w_ju)=\ell(w_j)+m+k$ and $\ell(w_jwy)=\ell(w_jw)+\ell(y)=\ell(w_j)+\ell(y)+m+k$.
\end{enumerate}
\end{lem}

\begin{proof} Part (1) of the lemma follows from \cite[Proposition 2.4(a)]{GP}. Note that both $y$ and $u$ commute with any element in $W_j$.
Thus Part (2) of the lemma follows from Lemma \ref{BM2}.
\end{proof}

The proof of the following lemma is given in the appendix of this paper.

\begin{lem}\label{r reduce} Let $m,\,k,\,j\in\Z_{\geq 0}, x\in\Sym_{j}$.
\begin{enumerate}
\item[(a)] For any $l\in \{1,\cdots,r-1\}$, we define $$\begin{aligned}
w(1)&:=s_{j+1}\cdots s_{j+m}s'_{j+m+1,l}s_{j+m+2}\cdots s_{j+m+k+1}x\\
v(1)&:=s'_{j,l}s_{j+1}\cdots s_{j+k}s_{j+k+2}\cdots s_{j+k+m+1}x.
\end{aligned}$$
\item[(b)] Assume $m>k\geq 0$. For any $l_1,\,l_2\in \{1,\cdots,r-1\}$, we define $$\begin{aligned}
w(2)&:=(s'_{j,l_1}s_{j+1}\cdots s_{j+m})(s'_{j+m+1,l_2}s_{j+m+2}\cdots s_{j+m+k+1})x\\
v(2)&:=(s'_{j,l_2}s_{j+1}\cdots s_{j+k})(s'_{j+k+1,l_1}s_{j+k+2}\cdots s_{j+k+m+1})x.
\end{aligned}$$
\item[(c)] Assume $m\geq 0$. For any $l_1,\,l_2\in \{1,\cdots,r-1\}$, we define $$\begin{aligned}
w(3)&:=(s'_{j,l_1}s_{j+1}\cdots s_{j+m})(s'_{j+m+1,l_2}s_{j+m+2}\cdots s_{j+2m+1})x\\
v(3)&:=(s'_{j,l_2}s_{j+1}\cdots s_{j+m})(s'_{j+m+1,l_1}s_{j+m+2}\cdots s_{j+2m+1})x.
\end{aligned}$$
\end{enumerate}
Let $c\in\{1,2,3\}$. There exists a sequence $s_{i_1},\cdots,s_{i_b}$ of standard generators in $\Sym_{\{j+1,j+2,\cdots,j+m+k+2\}}$ if $c\in\{1,2\}$, or in $\Sym_{\{j+1,j+2,\cdots,j+2m+2\}}$ if $c=3$, such that $$
w(c)=w_1\overset{s_{i_1}}{\rightarrow}w_2\overset{s_{i_2}}{\rightarrow}\cdots\overset{s_{i_b}}{\rightarrow}w_{b+1}=v(c). $$
\end{lem}

Henceforth, for each conjugacy class $C$ of $W$, we use $C_{\min}$ to denote the set of minimal length elements in $C$.

\begin{thm}\label{minimal length elements}
We have that  \begin{enumerate}
\item there exists a unique $\beta_C\in\CP_n$ such that $w_{\beta_C}\in C$. Moreover, $w_{\beta_C}\in C_{\min}$;
\item for any $w\in W$, there exists some $\alpha\in\CC_n$ such that $w\rightarrow w_\alpha$;
\item for any $\alpha\in\CC_n$, $w_\alpha$ is a minimal length element in its conjugacy class.
\end{enumerate}
\end{thm}

\begin{proof}
We divide the proof into three steps.

\medskip
{\it Step 1.} By Proposition \ref{reduce to composition}, for any $w\in W$, there exists a composition $\lambda=(\lambda_1,\cdots,\lambda_k)$ of $n$, and $\epsilon\in J^k$, such that $w \rightarrow w_{\lambda,\epsilon}$. Hence we reduce to the elements of the form $ w_{\lambda,\epsilon}$.

\medskip
{\it Step 2.} Let $\lambda=(\lambda_1,\cdots,\lambda_k)$ be a composition of $n$ and $\epsilon=(\epsilon_1,\cdots,\epsilon_k)\in J^k$ where $J=\{0,1,\cdots,r-1\}$. Let $1\leq l < k$. We set $$\begin{aligned}
& s_l\lambda:=(\lambda_1,\cdots,\lambda_{l+1},\,\lambda_l,\,\cdots\lambda_k), \,\,s_l\epsilon:=(\epsilon_1,\cdots,\epsilon_{l+1},\,\epsilon_l,\,\cdots,\epsilon_k),\\
& w_{\lam,\epsilon}^{\geq l+2}:=\Bigl(\prod_{i=l+2}^{k}w_{\lambda,\epsilon,i}\bigr).
\end{aligned}
$$
Now using the definition of $w_{\lambda,\epsilon,i}$ given in \eqref{wlam} and the defining relations of $W_n$, we can find some $x\in \Sym_{r_l}$ such that $$\begin{aligned}
&\prod_{i=1}^{k}w_{\lambda,\epsilon,i}=\Bigl(\prod_{i=1}^{l+1}w_{\lambda,\epsilon,i}\bigr)w_{\lam,\epsilon}^{\geq l+2}=
\Bigl(t_{r_1,\epsilon_1}\cdots t_{r_{l-1},\epsilon_{l-1}}w_{\lambda,\epsilon,l}w_{\lambda,\epsilon,l+1}x\Bigr)w_{\lam,\epsilon}^{\geq l+2},\\
&\prod_{i=1}^{k}w_{s_l\lambda,s_l\epsilon,i}=\Bigl(\prod_{i=1}^{l+1}w_{s_l\lambda,s_l\epsilon,i}\bigr)w_{\lam,\epsilon}^{\geq l+2}=\Bigl(t_{r_1,\epsilon_1}\cdots t_{r_{l-1},\epsilon_{l-1}}w_{\lambda,\epsilon,l+1}w_{\lambda,\epsilon,l}x\Bigr)w_{\lam,\epsilon}^{\geq l+2}.
\end{aligned}
$$
Using Corollary \ref{cancel12}, it is easy to see that $\ell(yw_{\lam,\epsilon}^{\geq l+2})=\ell(y)+\ell(w_{\lam,\epsilon}^{\geq l+2})$ for any $y\in W_{r_{l+2}}$. If $w_{\lam,\epsilon}=w_\alpha$ for some $\alpha\in\CC_n$, then we go to Step 3; otherwise we can find $1\leq l < k$ and $i\in\{1,2,3\}$, such that $$ w_{\lambda,\epsilon,l}w_{\lambda,\epsilon,l+1}x=w(i),\,w_{\lambda,\epsilon,l+1}w_{\lambda,\epsilon,l}x=v(i),
$$
where $v(i), w(i)$ are as defined in Lemma \ref{r reduce}. In this case we can use Lemma \ref{r reduce} and Corollary \ref{cancel12} to see that $$w_{\lambda,\epsilon}\rightarrow w_{s_l\lambda,s_l\epsilon}.
$$
Next, we replace $(\lam,\epsilon)$ with $(s_l\lam, s_l\epsilon)$ and repeat the argument from the beginning of Step 2. After finite steps, we can eventually show that $w_{\lambda,\epsilon}\rightarrow w_\alpha$ for some colored semi-bicomposition $\alpha=(\lam,\mu)\in\CC_n$.

\medskip
{\it Step 3.} It remains to show that each element $w_{\alpha}$, where $\alpha=(\lam,\mu)\in\CC_n$ with color $$
\epsilon=(\epsilon_1,\cdots,\epsilon_{\ell(\lam)},\underbrace{0,\cdots,0}_{\text{$\ell(\mu)$ copies}})\in J^{\ell(\lam)+\ell(\mu)},
$$ is a minimal length element in the conjugacy class of $w_{\alpha}$. Set $m:=|\lam|$ and $\epsilon(1)=(\epsilon_1,\cdots,\epsilon_{\ell(\lam)})$. In particular, $m\geq 1$. We can first decompose $w_{\alpha}=w_{\alpha,1}w_{\alpha,2}$, where $w_{\alpha,1}:=w_{\lam,\epsilon(1)}\in W_m$ corresponds to the opposite partition $\lam$, and $w_{\alpha,2}:=w_{\mu,(0,\cdots,0)}\in \Sym_{\{m+1,\cdots,n\}}$ corresponds to $\mu$.

Applying Lemma \ref{symmetric reduce} to $w_{\alpha,2}$, we can deduce that there exist $u_1,\cdots,u_b\in\Sym_{\{m+1,\cdots,n\}}$ such that  $v_{i+1}=u^{-1}_iv_{i}u_i$ and $\ell(v_i)=\ell(v_{i+1})$, for each $1\leq i<b$, and $v_0=w_{\alpha,2}$, $v_b=w_{\rho,(0,0,\cdots,0)}$ for some partition $\rho\in\mathscr{P}_{n-m}$. In particular, $\ell(w_{\alpha,2})=\ell(w_{\mu,(0,\cdots,0)})=\ell(w_{\rho,(0,0,\cdots,0)})$.
Our above proof from Step 1 to Step 3 implies that each conjugacy class $C$ of $W_n$ contains at least one element of the form $w_\beta$ with $\beta\in\CP_n$. On the other hand, it is well-known that the conjugacy classes of $W_n$ are in bijection with the set $\mathscr{P}_{r,n}$ of $r$-partition of $n$ (\cite[Remark 3.4]{Can}) and hence in bijection with the set $\CP_n$ by Lemma \ref{bije12}. It follows that each conjugacy class $C$ of $W_n$ contains a unique element of the form $w_\beta$ with $\beta\in\CP_n$. We denote it by $\beta_C$. Now we start from any minimal length element in the conjugacy class $C$, the above proof from Step 1 to Step 3 implies that $w_{\alpha}, w_{\beta_C}\in C_{\min}$. This proves Parts (1) and (2) of the theorem. Finally, the beginning of this paragraph proves that for each $\alpha\in\CC_n$, we can find a $\beta_C\in\CP_n$ such that $\ell(w_\alpha)=\ell(w_{\beta_C})$. Thus Part 3) of the theorem also follows.
\end{proof}

\medskip
\noindent
{\bf Proof of Theorem \ref{mainthm1}:} This follows from Theorem \ref{minimal length elements}.
\hfill\qed

\bigskip

\section{Cocenters of cyclotomic Hecke algebra}

The purpose of this section is to prove that the cocenter $\Tr(\HH_{n,R})$ is always a free $R$-module with an $R$-basis labelled by representatives of minimal length element in conjugacy classes when $R$ is commutative domain. As a consequence, we shall give a proof of Theorem \ref{mainthm2}.

Let $\HH_{n,R}$ be the cyclotomic Hecke algebra of type $G(r,1,n)$ with Hecke parameter $\xi\in R^\times$ and cyclotomic parameters $Q_1,\cdots,Q_r\in R$ and defined over a commutative (unital) ring $R$.

Let $w\in W_n$. If ${\bf w}=x_{i_1}\cdots x_{i_k}$ is a reduced expression of $w$, where $$
x_{i_1},\cdots,x_{i_k}\in\{t,s_1,\cdots,s_{n-1}\},$$ then we define $$
T_{\bf{w}}:=T_{x_{i_1}}\cdots T_{x_{i_k}}.
$$

\begin{lem}[\text{\cite[Prop 2.4]{BM}}]\label{expression}
For each $w\in W_n$, let ${\bf w},{\bf w'}$ be two reduced expressions of $w$, for any $y\in W_n$ with $\ell(y)<\ell(w)$, we fix a reduced expression $\bf{y}$ of $y$ and use it to define $T_{\bf{y}}$. Then $$T_{\bf{w}}-T_{\bf{w'}}\in \sum_{\substack{y\in W_n\setminus\Sym_n\\ 0<\ell(\bf{y})<\ell(w)}} RT_{\bf{y}}.
 $$
\end{lem}

By \cite{AK} we know that $\HH_{n,R}$ is a free $R$-module of rank $|W_n|$. For each $w\in W_n$, we fix a reduced expression ${\bf w}$ of $w$ and use it to define $T_{\bf w}$,  then it follows from Lemma \ref{expression} that $\{T_{\bf w}|w\in W_n\}$ forms an $R$-basis of $\HH_{n,R}$.

\begin{dfn} For each $\beta=(\lam,\mu)\in\CP_n$, we fix a reduced expression $\bf{w}_\beta$ of $w_\beta$ and define $T_{w_{\beta}}:=T_{\bf{w}_\beta}$.
\end{dfn}
Recall that ${\rm{Cl}}(W_n)$ denotes the set of conjugacy classes of $W_n$. For each $C\in {\rm{Cl}}(W_n)$, $C_{\min}$ denotes the set of minimal length elements in the conjugate class $C$.

\begin{thm}\label{span} Let $R$ be any commutative unital ring. As an $R$-module, we have \begin{equation}\label{generator1}
\Tr(\HH_{n,R})=\text{\rm $R$-Span}\bigl\{T_{w_{\beta}}+[\HH_{n,R},\HH_{n,R}]\bigm|\beta\in\CP_n\bigr\} .
\end{equation}
Moreover, for each conjugacy class $C$ of $W_n$, we arbitrarily choose an element $w_C\in C_{\min}$ and fix a reduced expression $x_1\cdots x_k$ of $w_{C}$, and define $T_{w_{C}}:=T_{x_1}\cdots T_{x_k}$, then \begin{equation}\label{generator0}
\Tr(\HH_{n,R})=\text{\rm $R$-Span}\bigl\{T_{w_{C}}+[\HH_{n,R},\HH_{n,R}]\bigm|C\in{\rm{Cl}}(W_n)\bigr\} .
\end{equation}
\end{thm}

\begin{proof} We first prove (\ref{generator1}). Set $$
\check{\HH}_{n,R}:=\text{\rm $R$-Span}\bigl\{T_{w_{\beta}}+[\HH_{n,R},\HH_{n,R}]\bigm|\beta\in\CP_n\bigr\}.
$$
We use induction on $\ell(w)$. The case $\ell(w)=0$ is clear, since $1=w_{\alpha}$ where $\alpha=(\emptyset,(1^n))\in\CP_n$. Suppose that for any $w\in W_{n}$ with $\ell(w)<m$ and any reduced expression ${\bf w}$ of $w$, we have $T_{\bf w}\in\check{\HH}_{n,R}$. Now we consider $w\in W_n$ with $\ell(w)=m$. By induction hypothesis and Lemma \ref{expression}, it suffices to show that there exists one reduced  expression ${\bf w}$ of $w$ such that $T_{\bf w}\in \check{\HH}_{n,R}$. The proof is divided into three steps as follows:

\medskip
{\it Step 1.} We fix a reduced  expression ${\bf w}$ of $w$ and define $T_w:=T_{\bf w}$. Consider the DC normal form of $w$ given in Proposition \ref{double coset} and (\ref{xadnb}), i.e., $$ w=ad_nb,
$$ where $d_n\in\D_n, a,b\in W_{n-1}$.  We first fix a reduced expression ${\bf w}(a)$ of $a$, a reduced expression ${\bf w}(d_n)$ of $d_n$, and define $$
T_a:=T_{{\bf w}(a)}, \quad T_{d_n}:=T_{{\bf w}(d_n)}.
$$
If $b\neq 1$ and ends with $s\in S\setminus\{s_{n-1}\}$, then we fix a reduced expression ${\bf w}(bs^{-1})$ of $bs^{-1}$ and define $T_{bs^{-1}}:=T_{{\bf w}(bs^{-1})}$. There are two cases:

\medskip
{\it Case 1.}  $s=t.$ If $a=t_{0,a_0}t_{1,a_1}\cdots t_{n-2,a_{n-2}}\sigma$ with $\sigma\in\Sym_{n-1}$ and $0\leq a_0<r-1$, then $\ell(ta)=\ell(a)+1$. Since
$\ell(w)=\ell(a)+\ell(d_n)+\ell(bt^{-1})+1$, it follows from induction hypothesis and Lemma \ref{expression} that $$
T_w\equiv T_{a}T_{d_n}T_{(bt^{-1})}T_{t}\equiv T_{t}T_{a}T_{d_n}T_{(bt^{-1})}\pmod{[\HH_{n,R},\HH_{n,R}]+\check{\HH}_{n,R}}.
$$
By construction, $\ell(w)=\ell(twt^{-1})=1+\ell(a)+\ell(d_n)+\ell(bt^{-1})$. It follows that $T_w\in\check{\HH}^{\Lam}_{n,R}$ if and only if for one (and hence any) reduced expression ${\bf w}(twt^{-1})$
of $twt^{-1}$, $T_{{\bf w}(twt^{-1})}\in\check{\HH}^{\Lam}_{n,R}$.

If $a_0=r-1$, then $ta=t_{1,a_1}\cdots t_{n-2,a_{n-2}}\sigma$ and hence $\ell(ta)=\ell(a)-(r-1)$. In this case, $$
T_{w}\equiv T_0^{r-1}T_{ta}T_{d_n}T_{bt^{-1}}T_{0}\equiv T_0^rT_{ta}T_{d_n}T_{bt^{-1}}\pmod{[\HH_{n,R},\HH_{n,R}]+\check{\HH}_{n,R}}.
$$
Using the cyclotomic relation $\prod_{i=1}^{r}(T_0-Q_i)=0$, we see that $$
T_0^rT_{ta}T_{d_n}T_{bt^{-1}}\in\text{$R$-Span}\Bigl\{T_{{\bf w}(u)}\Bigm|\begin{matrix}\text{$u\in W_n, \ell(u)<\ell(w)$, ${\bf w}(u)$ is a}\\
\text{reduced expression of $u$}\end{matrix}\Bigr\}.
$$
Applying induction hypothesis, we can deduce that $T_0^rT_{ta}T_{d_n}T_{bt^{-1}}\in\check{\HH}^{\Lam}_{n,R}$ and hence $T_w\in\check{\HH}^{\Lam}_{n,R}$ and we are done in this case.

{\it Case 2.}  $s=s_i$ for some $1\leq i<n-1$. In this case $\ell(ws^{-1})=\ell(ws)<\ell(w)$. If $\ell(sws)=\ell(w)$, then by Corollary \ref{cancel1} we see that $\ell(bs)=\ell(b)-1$ and $\ell(sa)=\ell(a)+1$. Note that ${\bf w}(a){\bf w}(d_n){\bf w}(bs)$ is a reduced expression of $ws$. We define $T_{ws}:=T_aT_{d_n}T_{bs}, T_{sa}=T_sT_a$. As ${\bf w}(a){\bf w}(d_n){\bf w}(bs)s$ is a reduced expression of $w$, we have $$
T_w\equiv T_{ws}T_s\equiv  T_{s}T_{ws}\equiv T_sT_{a}T_{d_n}T_{bs}\equiv T_{sa}T_{d_n}T_{bs}\pmod{[\HH_{n,R},\HH_{n,R}]+\check{\HH}_{n,R}}
$$ by induction hypothesis and Lemma \ref{expression} again.

If $\ell(sws)<\ell(w)$, then by Corollary \ref{cancel1} we can deduce that $\ell(sw)=\ell(w)-1=\ell(ws)$ and $\ell(w)=2+\ell(sws)$. In this case, we fix a reduced expression ${\bf w}(sws)$ of $sws$ then
$s{\bf w}(sws)s$ is a reduced expression of $w$. We define $T_{sws}:=T_{{\bf w}(sws)}$. Applying induction hypothesis and Lemma \ref{expression} again we can deduce that $$
T_{w}\equiv T_sT_{sws}T_{s}\equiv T_{sws}T^2_{s}\equiv T_{sws}((\xi-1)T_s+\xi) \pmod{[\HH_{n,R},\HH_{n,R}]+\check{\HH}_{n,R}}.
$$
As $\ell(sws)<\ell(w)$ and $\ell(sws)+1<\ell(w)$, it follows from induction hypothesis that $T_{sws}((\xi-1)T_s+\xi)\in\check{\HH}_{n,R}$ and hence $T_w\in\check{\HH}_{n,R}$ and we are done.

Repeating the application of the discussion in both Case 1 and Case 2, we can assume without no loss of generality that $b=1$. That says, $w=ad_n$.
Now we consider the $(W_{n-2},W_{n-2})$-double coset decomposition for $a\in W_{n-1}$ as in the proof of Proposition \ref{reduce to composition}, i.e., $$ a=a'd_{n-1}b',
$$
where $d_{n-1}\in\D_{n-1}, a',b'\in W_{n-2}$. Since $b'$ commutes with $d_{n}$, we can write $$w=a'd_{n-1}d_{n}b'.
$$
Now repeating the application of previous discussion in both Case 1 and Case 2, we can reduce to the case when $b'=1$. Next we consider the $(W_{n-3},W_{n-3})$-double coset decomposition of $a'\in W_{n-2}$ and repeating a similar argument at the beginning of this paragraph. After  finite steps, we see that there is no loss of generality to assume that $w=d_1d_2\cdots d_n$, where  $d_1\in\D_1,\cdots,d_n\in\D_n$ satisfying
$\ell(d_1)+\cdots+\ell(d_n)=m=\ell(w)$. Thus it suffices to show that $T_{d_1\cdots d_{n-1}d_n}\in\check{\HH}^{\Lam}_{n,R}$. Applying Lemma \ref{bije1}, we can find a composition $\rho=(\rho_1,\cdots,\rho_k)$ of $n$ and a sequence $\epsilon=(\epsilon_1,\cdots,\epsilon_k)\in J^k$ such that
$d_1\cdots d_n=w_{\rho,\epsilon}$. Thus we can assume without loss of generality that $w=w_{\rho,\epsilon}$.

\medskip
{\it Step 2.} Now we deal with the element $w=w_{\rho,\epsilon}$ as in the Step 2 of Theorem \ref{minimal length elements}. By Step 2 in the proof of Lemma \ref{minimal length elements}, we can choose the sequence $s_{j_1},\cdots,s_{j_b}\in\{s_1,s_2,\cdots,s_{n-1}\}$ such that in each step$$
w=w_{\rho,\epsilon}=w(1)\overset{s_{j_1}}{\longrightarrow} w(2)\overset{s_{j_2}}{\longrightarrow}\cdots\overset{s_{j_b}}{\longrightarrow}
w(b+1)=w_{\alpha},
$$
for some $\alpha=(\lam,\mu)\in\CC_n$. The main point here is, at each step since $s_{j_i}\in\{s_1,\cdots,s_{n-1}\}$, we have either $$
\ell(s_{j_i}w(i))=\ell(w(i))-1,\quad \ell(w(i)s_{j_i})=\ell(w(i))\pm 1 ;
$$
or $$
\ell(w(i)s_{j_i})=\ell(w(i))-1,\quad \ell(s_{j_i}w(i))=\ell(w(i))\pm 1 .
$$
Therefore, we can apply the same argument as in Step 1 to deduce that, in order to show $T_w=T_{w_{\rho,\epsilon}}\in\check{\HH}_{n,R}$, it suffices to show that for any $\alpha=(\lam,\mu)\in\CC_n$ with $\ell(w_\alpha)=\ell(w)$,
$T_{w_\alpha}\in\check{\HH}_{n,R}$. Thus we can assume without loss of generality that $w=w_\alpha$ for some $\alpha\in\CC_n$.

\medskip
{\it Step 3.} Finally, let $w=w_\alpha$, where $\alpha=(\lam,\mu)\in\CC_n$. As in the proof of Theorem \ref{minimal length elements},  we can decompose $w_{\alpha}=w_{\alpha,1}w_{\alpha,2}$, where $w_{\alpha,1}=w_{\lam,\epsilon(1)}\in W_m$ corresponds to the opposite partition $\lam\in\mathcal{P}_{m,-}$, $\epsilon(1)\in J^{\ell(\lam)}$ is as defined in Step 3 of the proof of Theorem \ref{minimal length elements},
and $w_{\alpha,2}=w_{\mu,(0,\cdots,0)}\in \Sym_{\{m+1,\cdots,n\}}$ corresponds to a composition $\mu$ of $n-m$. Applying Lemma \ref{symmetric reduce}, we can find $\hat{\rho}\in\mathcal{P}_{n-m}$, $w_{\alpha,2}=v_0, v_1, \cdots, v_l=w_{\hat{\rho},(0,\cdots,0)}\in \Sym_{\{m+1,\cdots,n\}}$, and $u_1,\cdots,u_s\in \Sym_{\{m+1,\cdots,n\}}$ such that \begin{enumerate}
\item[1)] $v_i=u^{-1}_iv_{i-1}u_i,\,\,\ell(v_{i-1}u_i)=\ell(v_{i-1})+\ell(u_i)$, $\forall\,1\leq i<l$; and
\item[2)] $\ell(v_i)=\ell(v_{i-1}),\,\forall\,1\leq i\leq l$.
\end{enumerate}
We want to show that
\begin{equation}\label{alphabeta}
T_{w_\alpha}\equiv T_{w_\beta}\pmod{[\HH_{n,R},\HH_{n,R}]}
\end{equation}
for some $\beta\in\CP_n$.

We first consider the case when $i=1$. The argument is somehow similar to the proof of \cite[Lemma 5.1]{HN}. We fix a reduced expression ${\bf w}(\alpha,1)$ (resp., ${\bf w}(\alpha)$) of $w_{\alpha,1}$ (resp., of $w_\alpha$) and define $T_{w_{\alpha,1}}:=T_{{\bf w}(\alpha,1)}$, $T_{w_{\alpha}}:=T_{{\bf w}(\alpha)}$. Note that for any $u\in\Sym_n$, one can use any reduced expression of $u$ to define $T_u$ and it depends only on $u$ but not on the choice of reduced expression because of the braid relations. Since $$
w_\alpha u_1=w_{\alpha,1}v_0u_1=w_{\alpha,1}u_1v_1.
$$
Note that $T_{w_{\alpha,1}}$ commutes with $T_{i}$ for any $m+1\leq i\leq n-1$ and $\ell(w_{\alpha,1})+\ell(u)=\ell(w_{\alpha,1}u)$ for any $u\in\Sym_{\{m+1,\cdots,n\}}$. We have the following {\it equalities}: $$
T_{w_{\alpha,1}}T_{v_0}T_{u_1}=T_{w_{\alpha,1}}T_{u_1}T_{v_1}=T_{u_1}T_{w_{\alpha,1}}T_{v_1} .$$
It follows that $$\begin{aligned}
T_{w_\alpha}&\equiv T_{w_{\alpha,1}}T_{v_0}\equiv T_{u_1}^{-1}T_{w_{\alpha,1}}T_{v_0}T_{u_1}\equiv T_{w_{\alpha,1}}T_{v_1}\\
&\equiv T_{w_{\alpha,1}v_1}\pmod{[\HH_{n,R},\HH_{n,R}]}.
\end{aligned}
$$
In the general case, one can show that for each $1\leq i\leq l-1$, $T_{w_{\alpha,1}v_i}\equiv T_{w_{\alpha,1}v_{i+1}}\pmod{[\HH_{n,R},\HH_{n,R}]}$. Since $w_{\alpha,1}v_l=w_{\alpha,1}w_{\hat{\rho},(0,\cdots,0)}=w_{(\lam,\hat{\rho})}\in\check{\HH}_{n,R}$, where $(\lam,\hat{\rho})\in\CP_n$. This completes the proof of (\ref{alphabeta}) and hence the first part of theorem.

Now for each conjugacy class $C$ of $W_n$ and $w\in C$, we claim that if $w\in C_{\min}$, and $\beta_C\in\P_n^c$ is the unique semi-bipartition such that $w_{\beta_C}\in C$, then
\begin{equation}\label{unitriangular1}
T_{w}\equiv T_{w_{\beta_C}}+\sum_{\substack{\beta\in\P_n^c\\ \ell(w_\beta)<\ell(w)}}a_{C,\beta}T_{w_\beta}\pmod{[\HH_{n,R},\HH_{n,R}]},
\end{equation}
where $a_{C,\beta}\in R$ for each $\beta$; while if $w\in C\setminus C_{\min}$, then \begin{equation}\label{unitriangular2}
T_{w}\equiv \sum_{\substack{\beta\in\P_n^c\\ \ell(w_\beta)<\ell(w)}}b_{C,\beta}T_{w_\beta}\pmod{[\HH_{n,R},\HH_{n,R}]},
\end{equation}
where $b_{C,\beta}\in R$ for each $\beta$. Once these two equalities are proved, the second part of the lemma follows immediately from  (\ref{unitriangular1}) and
(\ref{generator1}).

In fact, both (\ref{unitriangular1}) and (\ref{unitriangular2}) follows from an induction on $\ell(w)$, (\ref{alphabeta}), and a similar argument used in the Step 1 and Step 2 of the proof of (\ref{generator1}).
\end{proof}

Let $K$ be a field and $\xi\in K^\times, Q_1,\cdots,Q_r\in K$. Let $\O:= K[x]_{(x)}$, $\K := K(x)$, where $x$ is an indeterminate over $K$. Recall the definitions of the cyclotomic Hecke algebras $\HH_{n,K}, \HH_{n,\O}$ and $\HH_{n,\K}$ in Section 2.

\begin{lem}\label{lower} We have

1) $\dim Z(\HH_{n,\K})=\dim\Tr(\HH_{n,\K})=|\P_{r,n}|$;

2) $\dim Z(\HH_{n,K})\geq |\P_{r,n}|$.
\end{lem}

\begin{proof} Part 1) of the lemma is clear because $\HH_{n,\K}$ is isomorphic to a direct sum of some matrix algebras with $\{\mathfrak{f}_{\u\v}/\gamma_\u|\u,\v\in\Std(\blam),\blam\in\P_{r,n}\}$ being the set of matrix units. In fact, $Z(\HH_{n,\K})$ has a $\K$-basis $\{F_\bmu|\bmu\in\P_{r,n}\}$, and the following set $$
\bigl\{\mathfrak{f}_{\tlam\tlam}+[\HH_{n,\K},\HH_{n,\K}]\bigm|\blam\in\P_{r,n}\bigr\}
$$
is a $\K$-basis of $\Tr(\HH_{n,\K})$.

Since $\HH_n$ has an integral basis defined over $\O$, the calculation of $\dim Z(\HH_{n,\K})$
can be reduced to the calculation of the dimension of a solution space of a system of homogeneous linear equations with coefficient matrix $A$ defined over $\O$.
By general theory of linear algebras, the $\K$-rank of the matrix $A$ is bigger or equal to the $K$-rank of the matrix $1_K\otimes_{\O}A$, where $K$ is regarded as an $\O$-algebra by specializing $x$ to $0$. This proves that $$\dim Z(\HH_{n,K})\geq\dim Z(\HH_{n,\K})=|\P_{r,n}|.$$
Hence
$\dim\Tr(\HH_{n,K})=\dim Z(\HH_{n,\K})\geq |\P_{r,n}|$. This proves Part 2) of the lemma.
\end{proof}

Now we can give the proof of Theorem \ref{mainthm2}.
\medskip

\noindent
{\bf Proof of Theorem \ref{mainthm2}:}  Suppose $Q_1,\cdots,Q_r\in K^\times$. Then by \cite{MM}, $\HH_{n,K}$ is a symmetric algebra over $K$. By (\ref{ztr}), $Z(\HH_{n,\K})\cong(\Tr(\HH_{n,\K}))^*$. In particular,
$\dim Z(\HH_{n,\K})=\dim\Tr(\HH_{n,\K})$. For each conjugacy class $C$ of $W_n$, we arbitrarily choose an element $w_C\in C_{\min}$ and fix a reduced expression $x_1\cdots x_k$ of $w_{C}$, and define $T_{w_{C}}:=T_{x_1}\cdots T_{x_k}$. Combining Theorem \ref{span} and lemma \ref{lower}, we can deduce that $\dim Z(\HH_{n,\K})=\dim\Tr(\HH_{n,\K})=|\P_{r,n}|$ and the set \begin{equation}
\bigl\{T_{w_{C}}+[\HH_{n,K},\HH_{n,K}]\bigm|C\in{\rm Cl}(W_n)\bigr\}
\end{equation}
is in fact a $K$-basis of $\Tr(\HH_{n,K})$.

For any commutative domain $R$ with fraction field $F$, we have the following canonical map: $$
\psi: F\otimes_{R}\Tr(\HH_{n,R})\rightarrow \Tr(\HH_{n,F}) .
$$
Using Theorem \ref{span} and the fact $R\subseteq F$ it is easy to that the set (\ref{bas0}) is $R$-linearly independent and hence forms an $R$-basis of $\Tr(\HH_{n,R})$. In particular, $\Tr(\HH_{n,R})$ is a free $R$-module of rank $|\P_{r,n}|$. This proves Part 1) of Theorem \ref{mainthm2}.

Now combining (\ref{ztr}) and Part 1) of the theorem we can deduce that $Z(\HH_{n,R})$ is a free $R$-module of rank $|\P_{r,n}|$ too, and the dual $R$-basis of (\ref{bas0}) gives an $R$-basis of $Z(\HH_{n,R})$. Hence Part 2) of the theorem also follows.\qed
\medskip

\begin{rem}
	(1). The analog of the bases of the cocenter has been generalized to the degenerate cyclotomic Hecke algebra \cite[Theorem 5.6]{HuSS} and cyclotomic Sergeev algebra \cite[Theorem 1.3]{LS}.
	
		(2). By \cite[Proposition 2.1 (c)]{SVV}, for any $R$-algebra $A$ and any commutative domain $R'$ which is an $R$-algebra, we have $R'\otimes_R\Tr(A)\cong \Tr(R'\otimes_R A)$. That is, the cocenter is always stable under base change. However, this doesn't imply $\Tr(A)$ is free over $R$. Furthermore, in general, the $R$-module $[A,A]$ may not be stable unber base change. We shall prove $[\HH_{n,R},\HH_{n,R}]$ is stable under base change in Corollary \ref{purecor2}.
	
\end{rem}

\begin{cor}\label{astbasis} Let $R$ be a commutative domain and $\xi, Q_1,\cdots,Q_r\in R^\times$. For each conjugacy class $C$ of $W_n$, we arbitrarily choose an element $w_C\in C_{\min}$ and fix a reduced expression $x_1\cdots x_k$ of $w_{C}$, and define $T_{w_{C}}:=T_{x_1}\cdots T_{x_k}$. Then the set \begin{equation}
\bigl\{T_{w_{C}}^\ast+[\HH_{n,R},\HH_{n,R}]\bigm|\beta\in\CP_n\bigr\}
\end{equation}
is an $R$-basis of $\Tr(\HH_{n,R})$.
\end{cor}

\begin{proof} This is clear, because $\ast$ is an anti-isomorphism and $$
[\HH_{n,R},\HH_{n,R}]^*=[\HH_{n,R},\HH_{n,R}]. $$
\end{proof}

Let $R$ be a commutative ring and $M$ be a free $R$-module of finite rank. Recall that an $R$-submodule $N$ of $M$ is said to be $R$-pure if it satisfies that for any $y\in M$, $y\in N$ whenever $ry\in N$ for some $0\neq r\in R$. It is well-known that if $R$ is a principal ideal domain, then $M$ is a $R$-pure submodule of $N$ if and only if $M$ is an $R$-direct summand of $N$.

\begin{cor}\label{purecor2} Let $R$ be a commutative domain. Let $\xi, Q_1,\cdots,Q_r\in R^\times$. The $R$-submodule $[\HH_{n,R},\HH_{n,R}]$ is a pure $R$-submodule of $\HH_{n,R}$ of rank $r^n n!-|\P_{r,n}|$. Moreover, for any commutative domain $R'$ which is an $R$-algebra, the canonical map $$
R'\otimes_R[\HH_{n,R},\HH_{n,R}]\rightarrow [\HH_{n,R'},\HH_{n,R'}]
$$
is an $R'$-module isomorphism.
\end{cor}

\begin{proof} By Theorem \ref{mainthm2}, $\Tr(\HH_{n,R})=\HH_{n,R}/[\HH_{n,R},\HH_{n,R}]$ is a free $R$-module. Thus the short exact sequence \begin{equation}\label{split1}
0\rightarrow [\HH_{n,R},\HH_{n,R}]\rightarrow\HH_{n,R}\twoheadrightarrow \HH_{n,R}/[\HH_{n,R},\HH_{n,R}]\rightarrow 0 \end{equation}
must split. Hence the $R$-submodule $[\HH_{n,R},\HH_{n,R}]$ is a pure $R$-submodule of $\HH_{n,R}$ of rank $r^n n!-|\P_{r,n}|$.
The $R$-splitting of (\ref{split1}) implies that we again get a short exact sequence after tensoring with $R$: $$
0\rightarrow R'\otimes_R [\HH_{n,R},\HH_{n,R}]\rightarrow R'\otimes_R\HH_{n,R}\twoheadrightarrow R'\otimes_R \HH_{n,R}/[\HH_{n,R},\HH_{n,R}]\rightarrow 0.
$$
Now as $R'\otimes_R\HH_{n,R}\cong \HH_{n,R'}$ and by \cite[2.1(c)]{SVV}, $R'\otimes_R \HH_{n,R}/[\HH_{n,R},\HH_{n,R}]\cong \HH_{n,R'}/[\HH_{n,R'},\HH_{n,R'}]$. It follows that the canonical map
$R'\otimes_R[\HH_{n,R},\HH_{n,R}]\rightarrow [\HH_{n,R'},\HH_{n,R'}]$ is an isomorphism. This proves the corollary.
\end{proof}

\bigskip

\section{Class polynomials and integral polynomial coefficients}
In this section, we shall give the first application of our main result Theorem \ref{mainthm2}. Let $W$ be a real reflection group and ${\rm H}(W)$ be the associated Iwahori-Hecke algebra over ${\rm R}:=\Z[u_1^{\pm 1},\cdots,u_k^{\pm 1}]$ with Hecke parameters $u_1^{\pm 1},\cdots,u_k^{\pm 1}$, where $u_1,\cdots,u_k$ are indeterminates over $\Z$ and $k$ depends on $W$. Let $\{T_w|w\in W\}$ be the associated standard basis of ${\rm H}(W)$. Let ${{\rm{Cl}}(W)}$ be the set of conjugacy classes of $W$.
For each $C\in{{\rm{Cl}}(W)}$, we choose an element $w_C\in C$ such that $w_C$ is of minimal length in $C$. Geck and Pfeiffer have proved \cite{GP} (see also \cite{GP2}) that there exists a uniquely determined polynomial
$f_{w,C}$---the so-called {\it class polynomial}, which depends only on $w\in W$ and $C$ but not on the choice of the minimal length element $w_C$, such that $$
T_w\equiv \sum_{C\in{{\rm{Cl}}(W)}}f_{w,C}T_{w_C}\pmod{[{\rm H}(W),{\rm H}(W)]} .
$$
In other words, $\{T_{w_C}+[{\rm H}(W),{\rm H}(W)]|C\in{{\rm{Cl}}(W)}\}$ forms a basis of the cocenter of ${\rm H}(W)$.


Now return to the complex reflection group case. Let $W$ be a complex reflection group and $S$ be the set of distinguished pseudo-reflections of $W$. For each $s\in S$, let $e_s$ be the order of $s$ in $W$ and choose $e_s$ indeterminates $u_{s,1},\cdots,u_{s,e_s}$ such that $u_{s,j}=u_{t,j}$ if $s,t$ are conjugate in $W$. Malle has introduced in \cite{Mal} some indeterminates $v_{s,j}, s\in S, 1\leq j\leq e_s$, which are some $N_W$-th roots of scalar multiple of $u_{s,j}$, where $N_W\in\N$, see \cite[(2.7)]{CP}.
We set $$\begin{aligned}
& {\rm R}:=\Z[u_{s,j}^{\pm 1}|s\in S, 1\leq j\leq e_s],\quad {\rm F}:=\C(v_{s,j}|s\in S, 1\leq j\leq e_s).\\
& {\rm R}_{\mathbf{v}}:=\Z[v_{s,j}^{\pm 1}|s\in S, 1\leq j\leq e_s],\quad {\rm F}_{\mathbf{v}}:=\C[v_{s,j}^{\pm 1}|s\in S, 1\leq j\leq e_s].
\end{aligned}
$$
Let $\HH(W)$ be the associated generic cyclotomic Hecke algebra over ${\rm R}$ with parameters $\{u_{s,j}^{\pm 1}|s\in S, 1\leq j\leq e_s\}$. Malle \cite{Mal} showed that the $F$-algebra $F\otimes_{\rm R}\HH(W)$ is split semisimple. Specializing $v_{s,j}\mapsto 1$ for all $s\in S$ and $1\leq j\leq e_s$, then we get
$u_{s,j}\mapsto e^{2\pi\sqrt{-1}j/e_s}$ for all $s\in S$ and $1\leq j\leq e_s$, and a $\C$-algebra isomorphism \begin{equation}\label{specializ1}
\C\otimes_{{\rm F}_{\mathbf{v}}}{\rm F}_{\mathbf{v}}\otimes_{\rm R}\HH(W)\cong \C\otimes_{\rm R}\HH(W)\cong\C[W].
\end{equation} There is an ${\rm F}_{\mathbf{v}}$-algebra homomorphism \begin{equation}\label{specializ2}
\phi_{{\rm F}_{\mathbf{v}}}: {\rm F}_{\mathbf{v}}\otimes_{\rm R}\HH(W)\rightarrow{\rm F}_{\mathbf{v}}[W],
\end{equation}
such that $\id_{\C}\otimes_{{\rm F}_{\mathbf{v}}}\phi_{{\rm F}_{\mathbf{v}}}$ gives the isomorphism (\ref{specializ1}), and
$\id_{F}\otimes_{{\rm F}_{\mathbf{v}}}\phi_{{\rm F}_{\mathbf{v}}}$ defines an $F$-algebra isomorphism \begin{equation}\label{Tits}
{\rm F}\otimes_{\rm R}\HH(W)\cong {\rm F}[W],
\end{equation}
which is called Tits isomorphism.

We fix an ${\rm F}_{\mathbf{v}}$-basis $\mathcal{B}:=\{\fb_w|w\in W\}$ of $\HH(W)$ such that\footnote{This condition is implicit in the statement ``the specialization $v_{s,j}\mapsto 1$ induces a bijection ${\rm{Irr}}({\rm{H}}\otimes_{{\rm{R}}}F)\rightarrow{\rm{Irr}}(W)$'' in \cite[\S1]{CP} and needed in \cite[Theorem 3.2, 3.3]{CP}.}
$(\id_{\C}\otimes_{{\rm F}_{\mathbf{v}}}\phi_{{\rm F}_{\mathbf{v}}})(1_\C\otimes \fb_w)=w$ for each $w\in W$. For each $C\in{{\rm{Cl}}(W)}$, we fix a representative $w_C\in C$. Then $$\bigl\{\fb_{w_C}+[F\otimes_{\rm R}{\HH}(W),F\otimes_{\rm R}{\HH}(W)]\bigm|C\in{{\rm{Cl}}(W)}\bigr\}$$ forms a basis of the cocenter $\Tr({\rm F}\otimes_{\rm R}\HH(W))$ of ${\rm F}\otimes_{\rm R}{\HH}(W)$. Chavli and Pfeiffer (\cite{CP}) defined $f_{w,C}\in{\rm F}$ such that for any irreducible character $\chi$ of ${\rm F}\otimes_{\rm R}{\HH}(W)$, $$
\chi(\fb_w)=\sum_{C\in{{\rm{Cl}}(W)}}f_{w,C}\chi(\fb_{w_C}) .
$$
Equivalently, \begin{equation}\label{fwc}
\fb_w\equiv\sum_{C\in{{\rm{Cl}}(W)}}f_{w,C}\fb_{w_C}\pmod{[{\rm F}\otimes_{\rm R}{\HH}(W),{\rm F}\otimes_{\rm R}{\HH}(W)]} .
\end{equation}

Let $\{\fb_w^{\vee}|w\in W\}$ be the dual basis of $\mathcal{B}$ with respect to the symmetrizing form $\tau$. Chavli and Pfeiffer proved in
\cite[Theorem 3.2]{CP} that the following elements \begin{equation}\label{ycbasis1}
{\mathbf y}_C:=\sum_{w\in W}f_{w,C}\fb_{w}^\vee,\quad C\in{{\rm{Cl}}(W)}
\end{equation}
form an $F$-basis of the center $Z({\rm F}\otimes_{\rm R}\HH(W))$.

Chavli and Pfeiffer (\cite{CP}) also obtained a dual version of the above result.
Using the specialization map (\ref{specializ1}), one can see that $\{\fb_{w_C}^\vee+[{\rm F}\otimes_{\rm R}\HH(W),{\rm F}\otimes_{\rm R}\HH(W)]|C\in{{\rm{Cl}}(W)}\}$ forms an $F$-basis of the cocenter
$\Tr\bigl({\rm F}\otimes_{\rm R}\HH(W)\bigr)$. For each $w\in W$, we have $$
\fb_w^\vee\equiv\sum_{C\in{{\rm{Cl}}(W)}}g_{w,C}\fb_{w_C}^\vee\pmod{[{\rm F}\otimes_{\rm R}{\HH}(W),{\rm F}\otimes_{\rm R}{\HH}(W)]} ,
$$
where $g_{w,C}\in F$ for each pair $(w,C)$. Chavli and Pfeiffer proved in \cite[Theorem 3.3]{CP} that the following elements \begin{equation}\label{ycbasis2}
z_C:=\sum_{w\in W}g_{w,C}\fb_{w},\quad C\in{{\rm{Cl}}(W)}
\end{equation}
form an $F$-basis of the center $Z({\rm F}\otimes_{\rm R}\HH(W))$. They proposed the following conjecture.

\begin{conj}\text{\rm (\cite[Conjecture 3.7]{CP})}\label{cpconj} There exists a choice of an ${\rm R}$-basis $\{\fb_w|w\in W\}$ of the Hecke algebra $\HH(W)$, and a choice of conjugacy class representatives $\{w_C|C\in{{\rm{Cl}}(W)}\}$ such that $g_{w,C}\in R$ for each pair $(w,C)$, and hence $\{z_C|C\in{{\rm{Cl}}(W)}\}$ is an $R$-basis of $Z(\HH(W))$.
\end{conj}

In the rest of this section, we shall use our main result Theorem \ref{mainthm2} to verify Conjecture \ref{cpconj} for the cyclotomic Hecke algebra $\HH_{n,{\rm R}}=\HH(W)$ associated to the complex reflection group $W=W_n$ of type $G(r,1,n)$.

By \cite{MM}, $\HH_{n,{\rm R}}$ is a symmetric algebra over ${\rm R}$ with symmetrizing form $\tau_{\rm R}$. For each $u\in W_n$, we fix a reduced expression ${\bf u}$ and use this to define $T_u$.
Then Lemma \ref{expression} implies that $\{T_w|w\in W_n\}$ forms an ${\rm R}$-basis of $\HH_{n,{\rm R}}$. For this prefixed basis, let $$\{T^\vee_w|w\in W_n\}
$$ be its dual basis with respect to the symmetrizing form $\tau_{\rm R}$. For each $C\in{\rm{Cl}}(W_n)$, we arbitrarily choose an element $w_C\in C_{\min}$.
By Theorem \ref{mainthm2}(1), $\{T_{w_C}+[\HH_{n,{\rm R}},\HH_{n,{\rm R}}]|C\in{\rm{Cl}}(W_n)\}$ forms an ${\rm R}$-basis of the cocenter $\Tr(\HH_{n,{\rm R}})$.
Thus, for any $w\in W_n$,
\begin{equation}\label{fwc1}
T_w\equiv\sum_{\beta\in \CP_n}f_{w,C} T_{w_C}\pmod{[\HH_{n,{\rm R}},\HH_{n,{\rm R}}]},
\end{equation}
where $f_{w,C}\in{\rm R}$ for each $C\in{\rm{Cl}}(W_n)$. This proves the following result.

\begin{prop}\label{fwc2} For each $w\in W_n$, we define $b_w:=T_w$. For each $C\in{{\rm{Cl}}(W)}$, we arbitrarily choose an element $w_C\in C_{\min}$. Then the coefficient $f_{w,C}$ in (\ref{fwc}) lies in ${\rm R}$. Moreover, the set $$
\biggl\{{\mathbf y}_C=\sum_{w\in W_n}f_{w,C}T_w^\vee\biggm|C\in{\rm{Cl}}(W_n)\biggr\}
$$ forms an ${\rm R}$-basis of the center $Z(\HH_{n,{\rm R}})$.
\end{prop}

\begin{proof} The first part of the proposition follows from Theorem \ref{mainthm2}. For the second part of the proposition, by Theorem \ref{mainthm2}(1) and (\ref{fwc1}), we see that $f_{w_{C'},C}=\delta_{C,C'}$ for any $C,C'\in{\rm{Cl}}(W_n)$. It follows that $\{y_C|C\in{\rm{Cl}}(W_n)\}$ is the dual basis of the ${\rm R}$-basis
$\bigl\{T_{w_C}+[\HH_{n,{\rm R}},\HH_{n,{\rm R}}]\bigm|C\in{\rm{Cl}}(W_n)\bigr\}$ of the cocenter $\Tr(\HH_{n,{\rm R}})$ with respect to the isomorphism $Z(\HH_{n,{\rm R}})\cong\bigl(\Tr(\HH_{n,{\rm R}})\bigr)^*$. In particular, the set $\{{\mathbf y}_C|C\in{\rm{Cl}}(W_n)\}$ forms an ${\rm R}$-basis of the center $Z(\HH_{n,{\rm R}})$.
\end{proof}

\begin{rem} The above polynomial $f_{w,C}\in R$ is a natural generalization of Geck and Pfeiffer's class polynomial $f_{w,C}$. However, in contrast to the real reflection group case, for any two elements $w_1, w_2\in C_{\min}$, it may happen that $T_{w_1}\not\equiv T_{w_2}\pmod{[\HH_{n,{\rm R}},\HH_{n,{\rm R}}]}$ when $r>2$, as is shown in \cite[Example 4.4]{HuSS}. That says, $f_{w,C}$ may depends on the choice of elements in $C_{\min}$.
\end{rem}

Our next result verifies Conjecture \ref{cpconj} for the complex reflection group $W_n$ of type $G(r,1,n)$, which gives a second application of our main results.

\begin{prop}\label{apply2} Let $W=W_n$ be the complex reflection group of type $G(r,1,n)$. Then Conjecture \ref{cpconj} holds in this case.
\end{prop}

\begin{proof} For each $u\in W_n$, we fix a reduced expression ${\bf u}=x_1\cdots x_k$, where $x_i\in\{t,s_1,\cdots,s_{n-1}\}$ for each $i$, and use this to define $T_u:=T_{x_1}\cdots T_{x_k}$. Then Lemma \ref{expression} implies that $\{T_w|w\in W_n\}$ forms an ${\rm R}$-basis of $\HH_{n,{\rm R}}$.

Let $$\mathcal{B}:=\bigl\{\fb_w:=(T_{w^{-1}})^{\vee}\bigm|w\in W_n\bigr\}
$$ be the dual basis of $\{T_{w^{-1}}|w\in W_n\}$ with respect to the symmetrizing form $\tau_{\rm R}$. Note that the standard symmetrizing form $\tau_{\rm R}$ specializes to the standard symmetrizing form on $F[W]$ upon specializing $\xi\mapsto 1$ and $Q_j\mapsto e^{2\pi\sqrt{-1}j/r}$ for $1\leq j\leq r$. It follows that the basis $\mathcal{B}$ satisfies that $$
(\id_{\C}\otimes_{{\rm F}_{\mathbf{v}}}\phi_{{\rm F}_{\mathbf{v}}})(1_\C\otimes (T_{w^{-1}})^{\vee})=w,\quad \forall\,w\in W_n ,
$$
and $$
\bigl\{(T_{w_\beta^{-1}})^{\vee}+\bigl[\HH_{n,F},\HH_{n,F}\bigr]\bigm| \beta\in\P_n^c\bigr\}
$$
forms an $F$-basis of the cocenter $\Tr(\HH_{n,F})$.

Now the dual basis $\mathcal{B}^\vee$ of $\mathcal{B}$ is $\{\fb_w^{\vee}:=T_{w^{-1}}|w\in W_n\}$, which is an ${\rm R}$-basis of $\HH_{n,{\rm R}}$. It is clear that the basis $\mathcal{B}^\vee$ satisfies that $$
(\id_{\C}\otimes_{{\rm F}_{\mathbf{v}}}\phi_{{\rm F}_{\mathbf{v}}})(1_\C\otimes \fb_w^{\vee})=w^{-1},\quad \forall\,w\in W_n .
$$
Note that $\{w_\beta^{-1}|\beta\in\P_n^c\}$ is a complete set of representatives of conjugacy classes in $W_n$. It follows that for each $C\in{\rm{Cl}}(W_n)$, there is a unique $\hat{\beta}_C\in\P_n^c$ such that $w_{\hat{\beta}_C}^{-1}\in C$. We define $w_C:=w_{\hat{\beta}_C}^{-1}$. Then $$
\fb_{w_C}^\vee=T_{w_C^{-1}}=T_{w_{\hat{\beta}_C}} .
$$
Applying Theorem \ref{mainthm2}, $$
\bigl\{T_{w_{\hat{\beta}_C}}+[\HH_{n,{\rm R}},\HH_{n,{\rm R}}]\bigm|C\in{\rm{Cl}}(W_n)\bigr\}
$$
forms an ${\rm R}$-basis of the cocenter $\Tr(\HH_{n,{\rm R}})$. Therefore, for any $w\in W_n$, \begin{equation}\label{gwc1}
\fb_w^\vee\equiv\sum_{C\in{\rm{Cl}}(W_n)}g_{w,C}\fb_{w_C}^\vee\pmod{[\HH_{n,{\rm R}},\HH_{n,{\rm R}}]},
\end{equation}
where $g_{w,C}\in{\rm R}$ for each $C\in{\rm{Cl}}(W_n)$.

Finally, by construction, $z_C\in Z(\HH_{n,{\rm R}})$. We note that by Theorem \ref{mainthm2}(1) and (\ref{gwc1}), $g_{w_{C'},C}=\delta_{C,C'}$ for any $C,C'\in{\rm{Cl}}(W_n)$. It follows that $$
\biggl\{z_C=\sum_{w\in W}g_{w,C}\fb_{w}\biggm|C\in{\rm{Cl}}(W_n)\biggr\}
$$ is the dual basis of the ${\rm R}$-basis
$\bigl\{\fb_{w_C}^\vee+[\HH_{n,{\rm R}},\HH_{n,{\rm R}}]\bigm|C\in{\rm{Cl}}(W_n)\bigr\}$ of the cocenter $\Tr(\HH_{n,{\rm R}})$ with respect to the isomorphism $Z(\HH_{n,{\rm R}})\cong\bigl(\Tr(\HH_{n,{\rm R}})\bigr)^*$. In particular, the set $\{z_C|C\in{\rm{Cl}}(W_n)\}$ forms an ${\rm R}$-basis of the center $Z(\HH_{n,{\rm R}})$. This proves that Conjecture \ref{cpconj} holds in this case.
\end{proof}
\bigskip

\section{Cocenters and centers of the cyclotomic KLR algebras of affine type $A$}
In this section, we shall give the second application of our main result Theorem \ref{mainthm2}. We shall show that both the cocenters and the centers of certain cyclotomic KLR algebras of affine type $A$ are independent of the characteristic of the ground field $K$.

The KLR algebras and their cyclotomic quotients are some infinite family of $\Z$-graded algebras introduced by Khovanov and Lauda (\cite{KL1,KL2}), and by Rouquier (\cite{Rou2}). They are associated with a symmetrizable Cartan matrix, a family of polynomials $\{Q_{i,j}(u,v)\}$ and $\alpha\in Q_n^+$. Khovanov, Lauda  and Rouquier used these algebras to provide a categorification of the negative part of the quantum groups associated to the same Cartan datum and their integrable highest weight modules. These algebras play an important role in the categorical representation theory of $2$-Kac-Moody algebras. In this section, we shall only consider the cyclotomic KLR algebras associated to the Cartan datum of type $A_{e-1}^{(1)}$ (if $e>1$) or $A_{\infty}$ (if $e=0$), and the following choices of the polynomials $\{Q_{i,j}(u,v)\}$: \begin{equation}\label{qq}
Q_{i,j}(u,v):=\begin{cases} 0, &\text{if $i=j$;}\\
v-u, &\text{if $i\rightarrow j$;}\\
u-v, &\text{if $i\leftarrow j$;}\\
(u-v)^2, &\text{if $i\rightleftarrows j$;}\\
1, &\text{otherwise.}
\end{cases}\end{equation}
These cyclotomic KLR algebras of type $A$ will be closely related to the cyclotomic Hecke algebra.

Let $\Gamma_e$ be the quiver with vertex set $I:=\Z/e\Z$ and edges $i\rightarrow i+1$, for all $i\in I$. If $e=2$, then we shall write $i\rightleftarrows i+1$. Following \cite[Chapter 1]{Kac}, attach to $\Gamma_e$ the standard Lie theoretic data of a Cartan matrix $A=(a_{ij})_{i,j\in I}$, a set $\Pi=\{\alpha_i|i\in I\}$ of simple roots, a weight lattice $P$ which is a free abelian group with basis $\Pi$, a set $\Pi^\vee=\{\alpha_i^\vee|i\in I\}$ of simple co-roots, and the dual weight lattice $P^\vee:=\Hom_{\Z}(P,\Z)$ with basis $\Pi^\vee$, such that $\<\alpha_i^\vee,\alpha_j\>=a_{ij}$, $\forall\,i,j\in I$. We call $Q^+:=\sum_{i\in I}\mathbb{N}\alpha_i$ the positive root lattice. Let $\mathfrak{h}:=\mathbb{Q}\otimes_{\Z}P^\vee$. There is a symmetric bilinear form $(-,-)$ on $\mathfrak{h}^*$ such that $$
(\alpha_i,\alpha_j)=a_{ij},\,\, \<\alpha_i^\vee,\lam\>=2(\alpha_i,\lam)/(\alpha_i,\alpha_i)=(\alpha_i,\lam),\,\,\forall\,\lam\in\mathfrak{h}^*, i\in I .
$$
If $e>1$ then this Cartan datum is of type $A_{e-1}^{(1)}$; if $e=0$ then this Cartan datum is of type $A_{\infty}$.

Let $\{\Lam_j\in P|j\in I\}$ be the set of fundamental weights. That means, $\<\alpha_j^{\vee},\Lam_i\>=\delta_{ij},\forall\,i,j\in I$. We call $P^+:=\oplus_{j\in I}\mathbb{N}\Lam_j$ the set of dominant weights. For any  $\kappa_1,\cdots,\kappa_r\in\Z$, we set $$
\Lam:=\Lam_{\overline{\kappa_1}}+\cdots+\Lam_{\overline{\kappa_r}}\in P^+ ,
$$
where $\overline{\kappa_j}:=\kappa_j+e\Z\in I=\Z/e\Z$ for each $1\leq j\leq r$. We call $r$ the level of $\Lam$.

Let $1\neq\xi\in K^\times$ such that $\xi$ has quantum characteristic $e$. That means, either $e>1$ is the minimal positive integer $k$ such that $1+\xi+\xi^2+\cdots+\xi^{k-1}=0$; or $e:=0$ if no such positive integer $k$ exists.
Now we define the non-degenerate cyclotomic Hecke algebra \begin{equation}\label{HalphaLam}
\HH_{n,K}^\Lam:=\HH_{n,K}(\xi;\xi^{\kappa_1},\cdots,\xi^{\kappa_r}) .
\end{equation}

For each $\alpha=\sum_{i\in I}k_i\alpha_i\in Q^+$, we define $\Ht(\alpha):=\sum_{i\in I}k_i$. For each $n\in\N$, we set $Q_n^+:=\{\alpha\in Q_+|\Ht(\alpha)=n\}$. For any $\alpha\in Q_n^+$, we define $$
I^\alpha:=\biggl\{\nu=(\nu_1,\cdots,\nu_n)\in I^n\biggm|\sum_{i=1}^n\alpha_{\nu_i}=\alpha\biggr\}. $$
For each $\blam\in\P_{r,n}$ and $\gamma=(l,a,c)\in[\blam]$, we define $$
\res(\gamma):=\kappa_l+c-a+e\Z\in I=\Z/e\Z .
$$
We call $\gamma$ a $\res(\gamma)$-node. For each $\t\in\Std(\blam)$ and $1\leq k\leq n$, if $\t^{-1}(k)=\gamma$, then we define $\res_{\t}(k):=\res(\gamma)$. We set $$
\res(\t)=\bigl(\res_\t(1),\cdots,\res_\t(n)\bigr)\in I^n,\quad \bi^\blam:=\res({\tlam}).
$$ For any $\bi\in I^n$, we define $$
\Std(\bi):=\bigl\{\t\in\Std(\blam)\bigm|\blam\in\P_{r,n}, \res(\t)=\bi\bigr\}.
$$
We also write $\bi^\s:=\res(\s)$ for each $\s\in\Std(\blam)$. Finally, for $\alpha\in Q_n^+$, we define \begin{equation}\label{palphar}
\Parts[\alpha]^{(r)}:=\bigl\{\blam\in\P_{r,n}\bigm|\bi^\blam\in I^\alpha\bigr\}.
\end{equation}

\begin{dfn}\label{KLR}
Let $\alpha\in Q_n^+$. The type $A$ KLR algebra $\RR_{\alpha,K}$ is the unital associative algebra over $K$ generated by $$
y_1,\cdots,y_n, \psi_1,\cdots,\psi_{n-1}, e(\nu),\quad \nu\in I^\alpha ,
$$
satisfying the following defining relations:
  \begin{equation*}
    \begin{aligned}
      & e(\nu) e(\nu') = \delta_{\nu, \nu'} e(\nu), \ \
      \sum_{\nu \in I^{\alpha}}  e(\nu) = 1, \\
      & y_{k} y_{l} = y_{l} y_{k}, \ \ y_{k} e(\nu) = e(\nu) y_{k}, \\
      & \psi_{l} e(\nu) = e(s_{l}(\nu)) \psi_{l}, \ \ \psi_{k} \psi_{l} = \psi_{l} \psi_{k} \ \ \text{if} \ |k-l|>1, \\
      & \psi_{k}^2 e(\nu) =\begin{cases} 0, &\text{if $\nu_k=\nu_{k+1}$;}\\
      (y_{k}-y_{k+1})e(\nu), &\text{if $\nu_k\rightarrow \nu_{k+1}$;} \\
       (y_{k+1}-y_{k})e(\nu), &\text{if $\nu_k\leftarrow \nu_{k+1}$;} \\
      (y_{k+1}-y_k) (y_{k}-y_{k+1})e(\nu), &\text{if $\nu_k\rightleftarrows \nu_{k+1}$;}\\
      e(\nu), &\text{otherwise.}
      \end{cases}
       \\
      & (\psi_{k} y_{l} - y_{s_k(l)} \psi_{k}) e(\nu) = \begin{cases}
      -e(\nu) \ \ & \text{if $l=k, \nu_{k} = \nu_{k+1}$}, \\
      e(\nu) \ \ & \text{if $l=k+1, \nu_{k}=\nu_{k+1}$}, \\
      0 \ \ & \text{otherwise},
      \end{cases} \\[.5ex]
      & (\psi_{k+1} \psi_{k} \psi_{k+1}-\psi_{k} \psi_{k+1} \psi_{k})e(\nu)\\
      &\hspace*{8ex} =\begin{cases}
      e(\nu) & \text{if $\nu_{k} = \nu_{k+2}\rightarrow\nu_{k+1}$}, \\
      -e(\nu) & \text{if $\nu_{k} = \nu_{k+2}\leftarrow\nu_{k+1}$}, \\
      (2y_{k+1}-y_k-y_{k+2})e(\nu)  & \text{if $\nu_{k} = \nu_{k+2}\rightleftarrows\nu_{k+1}$}, \\
      0 \ \ & \text{otherwise},
      \end{cases}
    \end{aligned}
  \end{equation*}
for $\nu,\nu'\in I^\alpha$ and all admissible $k$ and $l$. For any $\Lam\in P^+$, let $\R[\alpha,K]$ be the quotient of $\RR_{\alpha,K}$ by the two-sided ideal generated by $$
y_1^{\<\alpha^\vee_{\nu_1},\Lam\>}e(\nu),\,\,\nu\in I^\alpha,
$$ where $\<,\>$ is the natural evaluation map associated to the Cartan datum of affine type $A$. The quotient algebra $\R[\alpha,K]$ is called the type $A$ cyclotomic KLR algebra.
\end{dfn}

Both $\RR_{\alpha,K}$ and $\R[\alpha,K]$ are $\Z$-graded with degree function determined by $$\deg e(\bi)=0,\qquad \deg y_a=2\qquad\text{and}\qquad \deg
  \psi_s e(\bi)=-a_{i_s,i_{s+1}},$$ for $1\le a\le n$, $1\le s<n$ and $\bi\in I^\alpha$, where $(a_{ij})_{i,j\in I}$ is the Cartan matrix of type $A_{e-1}^{(1)}$ (if $e>1$) or $A_{\infty}$ (if $e=0$).

The careful readers can notice that there are some sign differences with \cite[Definition 3.1]{HM} in the last three relations when $e\neq 2$. However, if $e\neq 2$, then this can be reconciled by consider the automorphism of $\R[\alpha,K]$ defined on KLR generators as follows: for any $\bi\in I^\alpha, 1\leq m\leq n, 1\leq r<n$, $$
e(\bi)\mapsto e(\bi),\,\, y_me(\bi)\mapsto y_me(\bi),\,\,\psi_re(\bi)\mapsto \begin{cases} -\psi_re(\bi), &\text{if $i_{r+1}=i_r+1$;}\\
\psi_re(\bi), &\text{otherwise.}
\end{cases}.$$

For each $\bi\in I^n$, there is a well-defined idempotent (possibly being $0$) $e(\bi)$ in $\HH_{n,K}^\Lam$ \cite[\S4.1]{BK:GradedKL}. For each $\alpha\in Q_n^+$, we define $e(\alpha):=\sum_{\bi\in I^\alpha}e(\bi)$. Then by \cite{LM}, $\{e(\alpha)|e(\alpha)\neq 0,\alpha\in Q_n^+\}$ is a complete set of pairwise orthogonal central primitive idempotents of $\HH_{n,K}^\Lam$. In other words, each nonzero two-sided ideal $\HH_{\alpha,K}^\Lam:=\HH_{n,K}^\Lam e(\alpha)$ is a block of $\HH_{n,K}^\Lam$.

\begin{thm}\text{\rm (\cite[Theorem 1.1]{BK:GradedKL})}\label{BKiso} Let $\alpha\in Q_n^+$, $\Lam=\Lam_{\overline{\kappa_1}}+\cdots+\Lam_{\overline{\kappa_r}}$, where $\kappa_1,\cdots,\kappa_r\in\Z$. Let $1\neq\xi\in K^\times$ such that $\xi$ has quantum characteristic $e$. Then there is an isomorphism of $K$-algebras $\theta_K^\Lam: \R[\alpha,K]\cong\HH_{\alpha,K}^\Lam$ that sends $e(\nu)$ to $e(\nu)$ for all $\nu\in I^\alpha$.
\end{thm}

\begin{rem} Note that the assumption that $\xi$ has quantum characteristic $e$ in Theorem \ref{BKiso} implies that $\cha K$ is coprime to $e$ whenever $e>1$ and $\cha K>0$.
\end{rem}

\begin{thm}\label{application2} Suppose that either $\cha K=0$; or $e=0$; or $\cha K>0$, $e>1$ and $\cha K$ is coprime to $e$. Then the dimensions of both the cocenter and the center of certain cyclotomic KLR algebra $\R[\alpha,K]$ are equal to
$|\Parts[\alpha]^{(r)}|$, which is independent of the characteristic of the ground field $K$.
\end{thm}

\begin{proof} If $F$ is a field extension of the field $K$, then using \cite[Proposition 2.1(c)]{SVV} and the fact that $Z(\R[\alpha,K])\cong (\Tr(\R[\alpha,K]))^*$, $Z(\R[\alpha,F])\cong (\Tr(\R[\alpha,F]))^*$ we can deduce that $\dim_K Z(\R[\alpha,K])=\dim_F Z(\R[\alpha,F])$.
If $\cha K=0$; or $e=0$; or $\cha K>0$, $e>1$ and $\cha K$ is coprime to $e$, then we can find a field extension $F$ of $K$ such that $F$ contains an element $\xi\in F$ which has quantum characteristic $e$. In this case, applying Theorem \ref{BKiso} we get that for any $\alpha\in Q_n^+$, \begin{equation}\label{cent11}
\dim_K Z(\R[\alpha,K])=\dim_F Z(\R[\alpha,F])=\dim_F Z(\HH_{\alpha,F}^\Lam).
\end{equation}
On the other hand, using \cite[Theorem A]{HuMathas:SeminormalQuiver}, we know that there is a modulo reduction system $(F,\O:=F[x],\K:=F(x))$, where $x$ is an indeterminate over $F$, and cyclotomic Hecke algebras $\HH_{\alpha,\O}^{\mathbf{\kappa}}, \HH_{\alpha,\K}^{\mathbf{\kappa}}$ such that $$
F\otimes_{\O}\HH_{\alpha,\O}^{\mathbf{\kappa}}\cong \HH_{\alpha,F}^\Lam,\,\,\K\otimes_{\O}\HH_{\alpha,\O}^{\mathbf{\kappa}}\cong\HH_{\alpha,\K}^{\mathbf{\kappa}},
$$
and $\HH_{\alpha,\K}^{\mathbf{\kappa}}$ is split semisimple over $\K$ with seminormal basis $\{\mf_{\s\t}\mid (\s,\t)\in \SStd(\blam), \blam\in \P_\alpha^{(r)}\}$. In particular, by Lemma \ref{fsemi},
$\dim_{\K}Z(\HH_{\alpha,\K}^{\mathbf{\kappa}})=|\Parts[\alpha]^{(r)}|$. Following a similar argument as in Lemma \ref{lower} (2), we have that $$
\dim_F Z(\HH_{\alpha,F}^\Lam)\geq \dim_{\K}Z(\HH_{\alpha,\K}^{\mathbf{\kappa}})=|\Parts[\alpha]^{(r)}|,\,\,\forall\,\alpha\in Q_n^+.
$$
However, by Theorem \ref{mainthm2}, we have that $$
\sum_{\alpha\in Q_n^+}\dim_F Z(\HH_{\alpha,F}^\Lam)=\dim_F Z(\HH_{n,F}^\Lam)=|\P_{r,n}|=\sum_{\alpha\in Q_n^+}|\Parts[\alpha]^{(r)}|.
$$
It follows that $\dim_F Z(\HH_{\alpha,F}^\Lam)=|\Parts[\alpha]^{(r)}|,\,\,\forall\,\alpha\in Q_n^+$. Combining this with (\ref{cent11}), we get that
$\dim_K Z(\R[\alpha,K])=\dim_F Z(\HH_{\alpha,F}^\Lam)=|\Parts[\alpha]^{(r)}|$, which independent of the characteristic of the ground field $K$.

As a result, we have that $$
\dim_K\Tr(\R[\alpha,K])=\dim_K Z(\R[\alpha,K])=|\Parts[\alpha]^{(r)}|, $$ which is independent of the characteristic of the ground field $K$.
\end{proof}

\begin{rem} If $e=\cha K>1$, then using Brundan-Kleshchev's isomorphism \cite[Theorem 1.1]{BK:GradedKL} in the degenerate setting and \cite[Theorem 1]{Brundan:degenCentre}, one can also deduce that the dimensions of both the cocenter and the center of certain cyclotomic KLR algebra $\R[\alpha,K]$ are equal to
$|\Parts[\alpha]^{(r)}|$.
\end{rem}

\bigskip
\section*{Appendix}

The purpose of this section is to give a proof of Lemma \ref{r reduce}.
\medskip

\begin{lem}[\text{\cite[Lemma 1.4]{BM}}]\label{BM1} Let $a,\,b\in\Z^{>0}$. We have the following equalities:$$\begin{aligned}
(1)\,\,&s_it_{k,a}=\begin{cases}t_{k,a}s_i, &\text{if $i>k+1 $;}\\
t_{k+1,a}, &\text{if $i=k+1$;}\\
t_{k,a}s_{i+1}, &\text{if $i<k$,}
\end{cases}\\
(2)\,\,&t_{k,a}t_{k,b}=\begin{cases}t_{k-1,b}t_{k,a}s_1, &\text{if $k>0 $;}\\
t_{0,a+b}, &\text{if $k=0$,}
\end{cases}\\
(3)\,\,&t_{k+m,a}t_{k,b}=t_{k-1,b}t_{k+m,a}s_1,\,\forall\, k>0,\,m\geq 0.
\end{aligned}
$$
\end{lem}

\noindent
{\textbf{Proof of Lemma \ref{r reduce}:}} The proof of both (a) and (b) are similar to \cite[Propsition 2.4 (b),(c)]{GP}. For the reader's convenience we include the details below.

(a). Let $l\in \{1,\cdots,r-1\}$. Using Lemma \ref{BM1}(1), we can write $$\begin{aligned}
w(1)&=t_{m+j+1,l}s_{j+2}\cdots s_{j+m+1}s_1\cdots s_{j+m+k+1} x\\
v(1)&=t_{j,l}s_1\cdots s_{j+k} s_{j+k+2}\cdots s_{j+k+m+1}x,
\end{aligned}$$ which are both BM normal forms \eqref{BM2}. Let $w_1:=w(1)$ and for $i=2,3,\cdots,k+1$, we set $$
w_i=t_{j,l}s_{m+i+j}\cdots s_{i+j}\cdots s_{m+i+j}s_{1}\cdots s_{m+k+j+1} x.
$$
If $k=0$ then $$
w(1)=s_{j+1}\cdots s_{j+m}s'_{j+m+1,l}x=s'_{j+m+1,l}s_{j+1}\cdots s_{j+m} x,\quad v(1)=s'_{j,l}s_{j+2,}\cdots s_{j+m+1}x.
$$
In this case, using braid relations and BM normal forms we see that $ w(1)\overset{(s_{m+1+j},\cdots, s_{1+j})}{\longrightarrow} v(1)$ and we are done. Henceforth, we
assume $k\geq 1$.

We claim that, for each $1 \leq i\leq k$,  \begin{equation}\label{i1} w_i\overset{(s_{m+i+j},\cdots, s_{i+j})}{\longrightarrow} w_{i+1}.
\end{equation}

We first consider the case when $i=1$. Assume $k\geq 1$. In this case, using braid relations, we see that $$
s_{m+j+1}w(1)s_{m+j+1}=t_{m+j,l}s_{j+2}\cdots s_{m+j+1}s_{m+j+2}s_1\cdots s_{m+k+j+1} x.
$$
Since $s_{m+j+1}w(1)=t_{m+j,l}s_{j+2}\cdots s_{m+j+1}s_1\cdots s_{m+k+j+1} x$ is still a BM normal form,  $\ell(s_{m+j+1}w(1))<\ell(w(1))$ by Lemma \ref{BM2}. Similarly, using braid relations, we have $$\begin{aligned}
&\quad\,s_{m+j}(s_{m+j+1}w(1)s_{m+j+1})s_{m+j}=s_{m+j}t_{m+j,l}s_{j+2}\cdots s_{m+j+2}s_1\cdots s_{m+k+j+1}s_{m+j} x\\
&=(t_{m+j-1,l}s_{j+2}\cdots s_{m+j+1}s_{m+j+2})s_{m+j+1}s_1\cdots s_{m+k+j+1} x\\
&=t_{m+j-1,l}s_{m+j+2}s_{j+2}\cdots s_{m+j+1}s_{m+j+2}s_1\cdots s_{m+k+j+1} x,
\end{aligned}
$$ and by the same argument as before, $$\ell(s_{m+j}(s_{m+j+1}w(1)s_{m+j+1}))<\ell((s_{m+j+1}w(1)s_{m+j+1}).
$$
In general, it follows from a similar argument that $$ w_1\overset{(s_{m+1+j},s_{m+j},\cdots, s_{1+j})}{\longrightarrow} w_{2}.
$$
This prove (\ref{i1}) for $i=1$,

Now assume $2\leq i\leq k$. Note that $s_{m+i+j}$ commutes with $t_{j,l}$. It follows again from braid relations that $$
s_{m+i+j}w_is_{m+i+j}=(t_{j,l}s_{m+i+j-1}\cdots s_{i+j}\cdots s_{m+i+j})s_{m+i+j+1}s_{1}\cdots s_{m+k+j+1} x.
$$ As $s_{m+i+j}w_i=t_{j,l}s_{m+i+j-1}\cdots s_{i+j}\cdots s_{m+i+j}s_{1}\cdots s_{m+k+j+1} x$ is a BM normal form, we can deduce from Lemma \ref{BM2} that $$
\ell(s_{m+i+j}w_i)<\ell(w_i).
$$
Similarly, using braid relations, for each $i+j\leq b\leq m+i+j-1$, we have
$$\begin{aligned}
&\quad\,s_b(s_{b+1}\cdots s_{m+i+j}w_is_{m+i+j}\cdots s_{b+1})s_b\\
&=s_b\bigl((t_{j,l}s_{m+i+j+1}\cdots s_{b+3})s_{b}s_{b-1}\cdots s_{i+j}\cdots s_{m+i+j}s_{m+i+j+1}s_1\cdots s_{m+k+j+1}\bigr)s_b x\\
&=(t_{j,l}s_{m+i+j+1}\cdots s_{b+3})s_{b+2}s_{b-1}\cdots s_{i+j}\cdots s_{m+i+j}s_{m+i+j+1}s_1\cdots s_{m+k+j+1} x.
\end{aligned}$$
Take $b=i+j$, we can get that $$\begin{aligned}
&\quad\,s_{i+j}(s_{i+j+1}\cdots s_{m+i+j}w_is_{m+i+j}\cdots s_{i+j+1})s_{i+j} \\
&=s_{i+j}\bigl((t_{j,l}s_{m+i+j+1}\cdots s_{i+j+3})(s_{i+j}s_{i+j+1}\cdots s_{m+i+j}s_{m+i+j+1})s_1\cdots s_{m+k+j+1}\bigr)s_{i+j} x\\
&=(t_{j,l}s_{m+i+j+1}\cdots s_{i+j+3})s_{i+j+1}\cdots s_{m+i+j}s_{m+i+j+1}s_{i+j+1}s_1\cdots s_{m+k+j+1} x\\
&=(t_{j,l}s_{m+i+j+1}\cdots s_{i+j+3})s_{i+j+2}s_{i+j+1}\cdots s_{m+i+j}s_{m+i+j+1}s_1\cdots s_{m+k+j+1} x=w_{i+1}.
\end{aligned}$$
Moreover, by the same argument, for each $i+j\leq b\leq m+i+j-1$, we have $$\begin{aligned}
&\ell(s_b(s_{b+1}\cdots s_{m+i+j}w_is_{m+i+j}\cdots s_{b+1}))=\ell(s_{b+1}\cdots s_{m+i+j}w_is_{m+i+j}\cdots s_{b+1})-1\\
&<\ell(s_{b+1}\cdots s_{m+i+j}w_is_{m+i+j}\cdots s_{b+1}) .
\end{aligned}$$
Finally, by a direct calculation one can see that \begin{equation}\label{afinal}
w(k+1)\overset{(s_{m+1+k+j},s_{m+k+j},\cdots, s_{1+k+j})}{\longrightarrow} v(1).
\end{equation}
This proves the lemma for $w(c)$ when $c=1$.

(b). As in (a), we can use braid relations to write $$\begin{aligned}
w(2)&=t_{j,l_1}t_{m+j+1,l_2}s_2\cdots s_{m+j+1}s_1\cdots s_{m+k+j+1} x\\
v(2)&=t_{j,l_2}t_{k+j+1,l_1}s_2\cdots s_{k+j+1}s_1\cdots s_{m+k+j+1} x,
\end{aligned}$$
where both of them are BM normal forms. If $k=0$, then $$
\begin{aligned}
w(2)&=t_{j,l_1}t_{m+j+1,l_2}s_2\cdots s_{m+j+1}s_1\cdots s_{m+j+1} x\\
v(2)&=t_{j,l_2}t_{j+1,l_1}s_2\cdots s_{j+1}s_1\cdots s_{m+j+1}=t_{j,l_2}t_{j+1,l_1}s_1\cdots s_{m+j+1}s_1\cdots s_{j} x.
\end{aligned}$$
In this case, using Lemma \ref{BM1}(2), it is easy to check that $$
w(2)\overset{(s_{m+j+1},\cdots,s_{j+1})}{\longrightarrow} v(2).
$$
Henceforth we assume $k\geq 1$.

Let $w_1=w(2)$ and for $2\leq i\leq k+1$ we set $$
w_i=t_{j,l_1}t_{m-k+j,l_2}s_{m+i+j}\cdots s_{m+2i+j-k-1}s_2\cdots s_{m+i+j}s_1\cdots s_{m+k+j+1} x.
$$
Now a completely similar computation as in (a) shows for any $1\leq i\leq k$, \begin{equation}\label{bwi}
w_i\overset{(s_{m+i+j},\cdots,s_{m-k+2i+j-1})}{\longrightarrow} w_{i+1} .
\end{equation}

Note that $$
w_{k+1}=t_{j,l_1}t_{m-k+j,l_2}s_{m+k+j+1}s_2\cdots s_{m+k+j+1}s_1\cdots s_{m+k+j+1} x.
$$
It is clear that $$w_{k+1}\overset{s_{m+k+j+1}}{\longrightarrow} x_{k+1}:=t_{j,l_1}t_{m-k+j,l_2}s_2\cdots s_{m+k+j+1}s_1\cdots s_{m+k+j} x,
$$ and $\ell(w_{k+1}s_{m+k+j+1})=\ell(w_{k+1})-1$ by Lemma \ref{BM2}. Next, for each $1\leq i\leq k$, we define $$
x_i=t_{j,l_1}t_{m-k+j,l_2}s_2\cdots s_{m+k+j+1}s_1\cdots s_{m+i+j-1}s_{m+i+j-2}\cdots s_{m-k+2i+j-2} x.
$$
We claim that for each $1\leq i\leq k$, \begin{equation}\label{xixi} x_{i+1}\overset{(s_{m-k+2i+j},\cdots,s_{m+i+j})}{\longrightarrow}x_i.
\end{equation}

First, using braid relations, we can check that $$\begin{aligned}
x_{k+1}&=t_{j,l_1}t_{m-k+j,l_2}s_2\cdots s_{m+k+j+1}s_1\cdots s_{m+k+j}x\overset{s_{m+k+j}}{\longrightarrow}\\
 &\qquad\qquad x_k=t_{j,l_1}t_{m-k+j,l_2}s_2\cdots s_{m+k+j+1}s_1\cdots s_{m+k+j}s_{m+k+j-2} x.
\end{aligned}
$$
In general, for $1\leq i\leq k-1$, we have $$\begin{aligned}
&\quad\,s_{m-k+2i+j}x_{i+1}s_{m-k+2i+j}\\
&=t_{j,l_1}t_{m-k+j,l_2}s_2\cdots s_{m+k+j+1}s_1\cdots s_{m+i+j}\cdots s_{m-k+2i+j+1}s_{m-k+2i+j-2} x
\end{aligned}
$$ and by Lemma \ref{BM2}, $$\ell(x_{i+1}s_{m-k+2i+j})<\ell(x_{i+1}).
$$
Similarly, $$\begin{aligned}
&\quad\,s_{m-k+2i+j+1}(s_{m-k+2i+j}x_{i+1}s_{m-k+2i+j})s_{m-k+2i+j+1}\\
&=s_{m-k+2i+j+1}t_{j,l_1}t_{m-k+j,l_2}s_2\cdots s_{m+k+j+1}s_1\cdots s_{m+i+j}\cdots s_{m-k+2i+j+1} \\
&\qquad s_{m-k+2i+j-2}s_{m-k+2i+j+1} x\\
&=t_{j,l_1}t_{m-k+j,l_2}s_2\cdots s_{m+k+j+1}s_1\cdots s_{m+i+j}\cdots s_{m-k+2i+j+2}s_{m-k+2i+j-1}\\
&\qquad s_{m-k+2i+j-2} x
\end{aligned}$$ and  $$
\ell((s_{m-k+2i+j}x_{i+1}s_{m-k+2i+j})s_{m-k+2i+j+1})<\ell(s_{m-k+2i+j}x_{i+1}s_{m-k+2i+j})
$$by Lemma \ref{BM2}.

Repeating this argument, we shall get that $$x_{i+1}\overset{(s_{m-k+2i+j},\cdots,s_{m+i+j})}{\longrightarrow}x_i.
$$

Now for $i=1,2,\cdots,m-k$, we define $$
v_i=t_{j,l_1}t_{i+j,l_2}s_2\cdots s_{m+k+j+1}s_1\cdots s_{k+i+j}\cdots s_{i+j} x.
$$
Note that $v_{m-k}=x_1$. By a direct calculation, one can check that for each $2\leq i \leq m-k$, $$ v_i\overset{(s_{i+j},\cdots,s_{i+k+j})}{\longrightarrow}v_{i-1}.
$$

Finally, we claim that $$
v_1\overset{(s_{j+1},\cdots,s_{j+1+k})}{\longrightarrow}t_{j,l_2}t_{k+j+1,l_1}s_1\cdots s_{m+k+j+1}s_1\cdots s_{k+j} x=v(2).
$$

In fact, we have $$\begin{aligned}
&\quad\, s_{j+k+1}s_{j+k}\cdots s_{j+1}v_1s_{j+1}\cdots s_{j+k}s_{j+k+1}\\
&=t_{j+1+k,l_1}t_{j+1,l_2}s_2\cdots s_{m+k+j+1}s_1\cdots s_{j+k} x\\
&=t_{j,l_2}t_{j+1+k,l_1}s_1s_2\cdots s_{m+k+j+1}s_1\cdots s_{j+k} x\quad  \text{(by Lemma \ref{BM1}(3))}\\
&=v(2) .
\end{aligned}
$$
Moreover, from the proof it is easy to check that each step of the above satisfies the requirement (\ref{2possibi}).
This proves $v_1\overset{(s_{j+1},\cdots,s_{j+1+k})}{\longrightarrow}v(2)$.

(c). Using braid relations, we can write $$\begin{aligned}
w(3)&=t_{j,l_1}t_{m+j+1,l_2}s_2\cdots s_{m+j+1}s_1\cdots s_{2m+j+1} x\\
v(3)&=t_{j,l_2}t_{m+j+1,l_1}s_2\cdots s_{m+j+1}s_1\cdots s_{2m+j+1} x.
\end{aligned}$$ We shall check $w(3)\rightarrow v(3)$. The case $m=0$ is easy since we have $$w(3)\overset{s_{j+1}}{\longrightarrow} v(3).
$$ From now on we suppose $m>0$. Let $w_1:=w(3)$ and for $2\leq i\leq m+1$, we define $$ w_i=t_{j,l_1}t_{j+1,l_2}s_2\cdots s_{m+i+j}\cdots s_{2i+j-1} s_1\cdots s_{2m+j+1} x. $$
We claim that for any $1\leq i\leq m$, \begin{equation}\label{claimwi1} w_{i}\overset{(s_{m+j+i},\cdots,s_{2i+j})}{\longrightarrow}w_{i+1}.
\end{equation}
Let's check this in detail. For $i=1$, we have $$
s_{m+j+1}w_1s_{m+j+1}=t_{j,l_1}t_{m+j,l_2}s_2\cdots s_{m+j+2}s_1\cdots s_{2m+j+1} x
$$ with $ \ell(s_{m+j+1}w)<\ell(w)$ by Lemma \ref{BM1} and Lemma \ref{BM2}. Similarly, $$\begin{aligned}
&\quad\, s_{m+j}(s_{m+j+1}w_1s_{m+j+1})s_{m+j}\\
&=s_{m+j}t_{j,l_1}t_{m+j,l_2}s_2\cdots s_{m+j+2}s_1\cdots s_{2m+j+1}s_{m+j} x\\
&=t_{j,l_1}t_{m+j-1,l_2}s_2\cdots s_{m+j+2}s_{m+j+1}s_1\cdots s_{2m+j+1} x
\end{aligned}
$$ and $\ell(s_{m+j}(s_{m+j+1}w_1s_{m+j+1}))<
\ell((s_{m+j+1}w_1s_{m+j+1}))$, by Lemma \ref{BM1} and Lemma \ref{BM2}. Continue in the same way we get $$
w_{1}\overset{(s_{m+j+i},\cdots,s_{2+j})}{\longrightarrow}w_{2}.
$$
Now let $i\geq 2$. We have $$s_{m+j+i}w_is_{m+j+i}=t_{j,l_1}t_{j+1,l_2}s_2\cdots s_{m+i+j+1} s_{m+i+j-2}\cdots s_{2i+j-1} s_1\cdots s_{2m+j+1} x
$$ and $\ell(s_{m+j+i}w_i)<\ell(w_i)$, by Lemma \ref{BM1} and Lemma \ref{BM2}. Similarly, $$\begin{aligned}
&\quad\,s_{m+j+i-1}(s_{m+j+i}w_is_{m+j+i})s_{m+j+i-1}\\
&=s_{m+j+i-1}t_{j,l_1}t_{j+1,l_2}s_2\cdots s_{m+i+j+1} s_{m+i+j-2}\cdots s_{2i+j-1} s_1\cdots s_{2m+j+1}s_{m+j+i-1} x\\
&=t_{j,l_1}t_{j+1,l_2}s_2\cdots s_{m+i+j+1} s_{m+i+j}s_{m+i+j-3}\cdots s_{2i+j-1} s_1\cdots s_{2m+j+1} x
\end{aligned}$$ and by Lemma \ref{BM1} and Lemma \ref{BM2}, $$
\ell(s_{m+j+i-1}(s_{m+j+i}w_is_{m+j+i}))<\ell(s_{m+j+i}w_is_{m+j+i}).
$$
Repeating a similar calculation, we can eventually verify that $$
s_{2i+j}\cdots (s_{m+j+i}w_is_{m+j+i})\cdots s_{2i+j}=w_{i+1}. $$ This proves our claim (\ref{claimwi1}).

Next, we set $v_{m+1}:=w_{m+1}$, and for $1\leq i\leq m$, we define $$ v_i=t_{j,l_1}t_{j+1,l_2}s_2\cdots s_{2m+1+j}s_{m+i+j}\cdots s_{2i+j}s_1\cdots s_{m+i+j} x.$$
We claim that for each $2\leq i\leq m+1$, \begin{equation}\label{vi2} v_{i}\overset{(s_{2i+j-1},\cdots, s_{m+i+j})}{\longrightarrow}v_{i-1}.
\end{equation}

In fact, take $i=m+1$, we have $$s_{2m+j+1}v_{m+1}s_{2m+j+1}=t_{j,l_1}t_{j+1,l_2}s_2\cdots s_{2m+j}s_{2m+j+1} s_{2m+j}s_1\cdots s_{2m+j} x=v_m,
$$ and by Lemma \ref{BM1} and Lemma \ref{BM2}, $$\ell(v_{m+1}s_{2m+j+1})<\ell(v_{m+1}).
$$
This proves (\ref{vi2}) for $i=m+1$.

For $2\leq i\leq m$, we have $$s_{2i+j-1}v_{i}s_{2i+j-1}=t_{j,l_1}t_{j+1,l_2}s_2\cdots s_{2m+1+j}s_{m+i+j}\cdots s_{2i+j+1}s_{2i+j-2}s_1\cdots s_{m+i+j} x
$$ and by Lemma \ref{BM1} and Lemma \ref{BM2}, $\ell(v_{i}s_{2i+j-1})<\ell(v_{i})$.  Similarly, we can compute $$\begin{aligned}
&\quad\,s_{2i+j}(s_{2i+j-1}v_{i}s_{2i+j-1})s_{2i+j}\\
&=s_{2i+j}t_{j,l_1}t_{j+1,l_2}s_2\cdots s_{2m+1+j}s_{m+i+j}\cdots s_{2i+j+1}s_{2i+j-2}s_1\cdots s_{m+i+j}s_{2i+j} x\\
&=t_{j,l_1}t_{j+1,l_2}s_2\cdots s_{2m+1+j}s_{m+i+j}\cdots s_{2i+j+2}s_{2i+j-1}s_{2i+j-2}s_1\cdots s_{m+i+j} x\end{aligned}
$$ and  $$
\ell(s_{2i+j-1}v_{i}s_{2i+j-1})s_{2i+j})<\ell(s_{2i+j-1}v_{i}s_{2i+j-1})
$$ by Lemma \ref{BM1} and Lemma \ref{BM2}. Repeating a similar calculation, we can eventually verify that $$
s_{m+i+j}\cdots (s_{j+2i-1}v_is_{j+2i-1})\cdots s_{m+i+j}=v_{i-1}.$$ This proves our claim (\ref{vi2}).

Finally, by a similar calculation, one can verify that $$v_1\overset{(s_{j+1},\cdots,s_{m+j+1})}{\longrightarrow}t_{m+j+1,l_1}t_{j+1,l_2}s_2\cdots s_{2m+j+1}s_1\cdots s_{m+j} x,
$$
and each step of the above satisfies the requirement (\ref{2possibi}). Now applying Lemma \ref{BM1}(3), we see that $$\begin{aligned}
&\quad\,t_{m+j+1,l_1}t_{j+1,l_2}s_2\cdots s_{2m+j+1}s_1\cdots s_{m+j} x\\
&=t_{j,l_2}t_{m+j+1,l_1}s_1(s_2\cdots s_{2m+j+1})(s_1\cdots s_{m+j}) x\\
&=t_{j,l_2}t_{m+j+1,l_1}s_2\cdots s_{m+j+1}s_1\cdots s_{2m+j+1} x=v(3) .
\end{aligned}
$$
This completes the proof of the lemma.\qed

\end{document}